\documentclass[10pt]{article}
\usepackage{amssymb,amsmath,amsthm,amsfonts}
\usepackage{cite}
\usepackage{bigints}
\usepackage[colorlinks=true,linkcolor=blue,urlcolor=blue,citecolor=blue]{hyperref}
\usepackage[margin=100pt]{geometry}

\usepackage[pagewise,mathlines]{lineno}
\usepackage{fancyhdr}
\tt
\theoremstyle{plain}
\newtheorem{theorem}{Theorem}
\newtheorem{lemma}{Lemma}
\newtheorem{proposition}{Proposition}
\newtheorem{corollary}{Corollary}
\theoremstyle{definition}
\newtheorem{definition}{Definition}
\theoremstyle{remark}
\newtheorem{remark}{Remark}

\newcommand{\esssup}{{\mathrm{ess}\sup\,}}
\newcommand{\essinf}{{\mathrm{ess}\inf\,}}

\begin{document}

\title{\bf\Large Maximization and minimization of the principal eigenvalue of the Laplacian with indefinite weight
\\under Dirichlet and Robin boundary conditions\\ on classes of rearrangements}

\author{Claudia Anedda\footnote{Department of Mathematics and Computer Science,
University of Cagliari, Via Ospedale 72, Cagliari, 09124,  Italy (\tt canedda@unica.it).}\; and\;
Fabrizio Cuccu\footnote{Department of Mathematics and Computer Science,
University of Cagliari, Via Ospedale 72, Cagliari, 09124,  Italy (\tt fcuccu@unica.it).}}

\maketitle
\begin{abstract}
\noindent Let $\Omega\subset\mathbb{R}^N$, $N\geq 1$, be a bounded connected open set. We consider 
the weighted  eigenvalue problem   $-\Delta u =\lambda m u$ in $\Omega$ with $\lambda \in \mathbb{R}$,
$m\in L^\infty(\Omega)$  and with homogeneous Dirichlet and Robin boundary conditions.
First, we study weak* continuity, convexity and G\^ateaux differentiability of   
the map $m\mapsto1/\lambda_1(m)$, where $\lambda_1(m)$ is the principal eigenvalue. 
Then, denoting by $\mathcal{G}(m_0)$ the class of rearrangements of a fixed weight $m_0$
and assuming that $m_0$ is positive on a set of positive Lebesgue measure, we investigate the
minimization and maximization of $\lambda_1(m)$ over $\mathcal{G}(m_0)$. The minimization 
problem has been already discussed in some papers; here we prove some known results about the existence 
and characterization of minimizers of $\lambda_1(m)$. We underline that our approach allows us to treat
Dirichlet and Robin boundary conditions together.
Instead, to our best knowledge, the maximization problem has been only partially addressed in the literature.  
It turns out that the maximization of $\lambda_1(m)$ is more intricate than its minimization.
In our work  we discuss  existence, uniqueness and characterization of maximizers both in $
\mathcal{G}(m_0)$ and in its weak*  closure $\overline{\mathcal{G}(m_0)}$. In particular, we provide an
original full description  of the unique maximizer in the case of Dirichlet boundary conditions.
In the context of the population dynamics, this kind of problems arise from the 
question of determining the optimal spatial location of favourable and unfavourable habitats 
in order to increase the chances of survival or extinction of a population.  
\end{abstract}



\section{Introduction and main results}\label{intro}

In this paper we consider the following indefinite weighted eigenvalue problem   
\begin{equation}\label{p0}
\begin{cases}-\Delta u =\lambda m u \quad &\text{in } \Omega\\
\cfrac{\partial u}{\partial \nu} +\sigma(x) u=0  &\text{on } \partial\Omega, \end{cases}  
\end{equation}  
where $\Omega\subset\mathbb{R}^N$ is a bounded domain with Lipschitz boundary $\partial\Omega$, 
 $m\in L^\infty(\Omega)$, $\lambda\in\mathbb{R}$, $\nu$ is the outward unit normal vector on $\partial \Omega$ and $\sigma(x) \in C(\partial \Omega)$ is a fixed function,
$\sigma(x)\geq 0$. Note that, when $\sigma(x)\equiv +\infty$ 
we obtain the Dirichlet boundary conditions, if 
 $\sigma(x)\equiv 0$ we have the Neumann boundary conditions and when $\sigma(x)$ is not identically zero and 
 finite we get the Robin boundary 
 conditions. Here, we focus on the Dirichlet and Robin boundary conditions only, the case of the Neumann boundary conditions
 has been investigated in \cite{AC23}.\\ 
An eigenvalue $\lambda$ of \eqref{p0} is called {\it principal eigenvalue} if it admits a positive eigenfunction.
Problem \eqref{p0} has been studied in various papers (see, for example, \cite{AB,Bo,BL}). In particular, it is known that there is a positive principal eigenvalue if and only if the Lebesgue measure of the set $\{x\in\Omega:m(x)>0\}$ is positive.
For the sake of completeness and in order to maintain this paper self-contained, we prefer to give here
(see Section \ref{preliminaries}) an independent proof of the result above. Moreover, we show that, under the previous hypothesis on $m$, there exists an increasing sequence of     
positive eigenvalues. The smallest positive eigenvalue is the principal eigenvalue, which will be denoted by $
\lambda_1(m)$.\\   
Problem \eqref{p0} and its variants play a crucial role in studying nonlinear models from 
population  dynamics (see \cite{S}) and population genetics (see \cite{F}).\\      
We illustrate in details the following model in population dynamics analyzed in \cite{CC}
\begin{equation}\label{p1}
\begin{cases}
v_t=d\Delta v+v[m(x)-cv] \quad &\text{in }\Omega\times(0,\infty),\\
\cfrac{\partial v}{\partial \nu} +\sigma(x) v=0 & \text{on } \partial\Omega\times(0,\infty),\\
v(x,0)=v_0(x)\geq 0, \  v(x,0)\not \equiv 0& \text{in } \overline{\Omega}.
\end{cases}  
\end{equation}

In \eqref{p1} $v(x,t)$ represents the density of a population inhabiting the region $\Omega$ at location $x$ and time 
$t$ (for that reason, only non-negative solutions of \eqref{p1} are of interest), $v_0(x)$ is the initial density, $d$ is 
the diffusion rate of the population and $c$ is a constant describing the limiting effects of crowding. The function $m(x)
$ represents the intrinsic local growth rate of the population, it is positive on favourable habitats and negative on 
unfavourable ones; it mathematically describes the available resources in the spatially heterogeneous environment
$\Omega$.  More precisely, by the term ``resource'' we mean anything that affects the growth rate of the population both in positive (for example food) and in negative (for example predators) way. The integral $\int_\Omega m\, dx$ can be interpreted as a measure of the total resources in $\Omega$.
The boundary conditions have the following biological meaning.
When $\sigma(x)\equiv+\infty$ (i.e. Dirichlet boundary conditions), $\partial\Omega$ is a deadly boundary:  the 
environment $\Omega$ is surrounded by a completely inhospitable region, i.e. any individual reaching the boundary 
dies.
If $\sigma(x)\equiv 0$ (i.e. Neumann boundary conditions),  the boundary acts as a barrier, i.e. any individual reaching the 
boundary returns to the interior.
Finally, when $\sigma(x)$ is not identically zero and 
finite (i.e. Robin boundary conditions), $\Omega$ is surrounded by a hostile but not deadly 
region (the larger $\sigma(x)$ is, the more inhospitable the region is):  some members of the population reaching the 
boundary of $\Omega$ die and others turn back.

It is known (see \cite{CC,CC91} and references therein) that \eqref{p1} predicts persistence for the population if $
\lambda_1(m)<1/d$  and extinction if $\lambda_1(m)\geq 1/d$. As a consequence, determining the best spatial 
arrangement of favourable and unfavourable habitats for the survival (resp. for the extinction), within a fixed 
class of environmental configurations, results in minimizing (resp. maximizing) $\lambda_1(m)$ over the 
corresponding class of weights. 
Having information of this type could affect, for example, the strategies to be adopted for the conservation of 
a population with limited resources or extermination of some pest population with an amount of pesticide enough to 
treat only some fraction of the infested region.\\
This kind of problem has been investigated by many other authors. The question of determining the optimal spatial arrangements of favourable and unfavourable habitats in $\Omega$ for the survival of the modelled population
was first addressed by Cantrell and Cosner in \cite{CC, CC91}, where the authors considered the 
diffusive logistic equation \eqref{p1} with homogeneous Dirichlet boundary conditions and when the weight $m$ has fixed maximum, minimum and integral over $\Omega$. The analogous problem with Neumann boundary conditions has been analysed
by Lou and Yanagida in \cite{LY}. Berestycki et al. (\hspace{-0.01cm}\cite{BHR}) investigated how the fragmentation of the environment affects the persistence of the population in the case of periodically fragmented environment ($\Omega=\mathbb{R}^N$ and $m(x)$ periodic), Roques and Hamel (\hspace{-0.01cm}\cite{RH}) studied the optimal arrangement of resources by using numerical computation, Jha and Porru (\hspace{-0.01cm}\cite{JP}), among the other things, exhibited an example of symmetry breaking of the optimal arrangement of the local growth rate in the case of Neumann boundary conditions. Hinterm\"uller et al. (\hspace{-0.01cm}\cite{HKL}) and Lamboley et al. 
(\hspace{-0.01cm}\cite{LLNP}) investigated model \eqref{p1} with Robin boundary conditions. Mazari et al. (\hspace{-0.01cm}\cite{MNP}) studied 
several shape optimization problems arising in population dynamics.
Finally, in \cite{AC} and in \cite{AC23} it is shown the Steiner symmetry of the optimal
rearrangement under Dirichlet boundary conditions and the monotonicity property in cylindrical domains under Neumann
boundary conditions, respectively.  \\ 
We observe that  problem \eqref{p0} with weight $m(x)$ positive and Dirichlet boundary conditions
also has a well known physical interpretation: it models  the normal modes of vibration of a membrane 
$\Omega$ with clamped boundary $\partial\Omega$ and mass
density $m(x)$; $\lambda_1(m)$ represents the principal natural frequency of the membrane. 
Therefore, physically, minimizing (resp. maximizing) $\lambda_1(m)$ means finding the mass distribution of the 
membrane which gives the lowest (resp. highest) principal natural frequency. 
Usually, the membrane is built by using only two homogeneous materials of 
different densities  and, then, the weights in the optimization problem take only two positive 
values. Among many papers that consider the optimization of the principal natural frequency, we recall
\cite{CGIKO,CML1,CML2}. \\
In order to present our work, we briefly give some notations and definitions. 
We denote by $\lambda_k(m)$, $k\in \mathbb{N}$, the $k$-th positive eigenvalue of problem
 \eqref{p0} corresponding to the weight $m$. We say that
 two Lebesgue measurable functions $f,g:\Omega \to \mathbb{R}$
are \emph{equimeasurable} if the superlevel sets
$\{x\in \Omega: f(x)>t\}$ and $\{x\in \Omega: g(x)>t\}$ have the same  measure for all $t\in \mathbb{R}$. For a fixed $f\in L^\infty(\Omega)$, we call the set
$\mathcal{G}(f)=\{g:\Omega\to\mathbb{R}: g \text{ is measurable and $g$ and
$f$ are equimeasurable}\}$ the
\emph{class of rearrangements of $f$} (see Section
 \ref{rearrangements}).\\
 The present paper contains the following main results. First, we study the dependence of $
 \lambda_k(m)$ on  $m$, in 
 particular we investigate
 continuity and, for $k=1$, convexity and differentiability properties (see Lemmas \ref{teo1},
 \ref{teo2} and \ref{teo3}).
Then, we examine the optimization of $\lambda_1(m)$ in the class of rearrangements $\mathcal{G}
(m_0)$  of a fixed function $m_0\in L^\infty(\Omega)$.
Precisely, we prove the existence of minimizers and a characterization of them in terms of the 
 eigenfunctions relative to $\lambda_1(m)$ (see Theorem \ref{exist} below).
 Regarding the maximization, we give conditions for the
 existence and uniqueness of  maximizers, furthermore we discuss their existence and provide a 
 characterization in the closure 
 $\overline{\mathcal{G}(m_0)}$ of $\mathcal{G}(m_0)$ with respect to the weak*  topology of 
 $L^\infty(\Omega)$ (see Theorem \ref{existmax} below).   
 We note that, considering classes of rearrangements allows us to deal with the class of weights
 usually considered in literature, i.e. a set of  bounded functions with 
 fixed maximum, minimum and integral over $\Omega$ as a particular case (see, for example, \cite{CC,LY, LLNP}). 
 Indeed, this class can be written in terms of   for a weight $m_0$ which takes exactly two values (functions of this kind are called of ``bang-bang'' type), for details see \cite{AC}.

\begin{theorem}\label{exist} 
 Let $\lambda_1(m)$ be the principal eigenvalue of problem
\eqref{p0},   
 $m_0\in L^\infty(\Omega)$      
such that the set $\{x\in \Omega:m_0(x)>0\}$ has positive Lebesgue measure and $\mathcal{G}(m_0)$ the class of rearrangements
 of $m_0$ (see Definition \ref{class}). Then \\
 i) there exists a solution of the problem   
 \begin{equation}\label{infclos0}
 \min_{m\in{\mathcal{G}(m_0)}} \lambda_1(m);
 \end{equation}  \\
 ii) for every solution $\check{m}_1\in\mathcal{G}(m_0)$ of \eqref{infclos0}, 
  there exists an increasing function $\psi$ such
 that      
 \begin{equation}\label{carat0}\check{m}_1= \psi(u_{\check{m}_1}) \quad \text{a.e. in }\Omega, 
 \end{equation} where 
$u_{\check{m}_1}$ is 
 the unique positive eigenfunction relative to $\lambda_1(\check{m}_1)$ normalized 
as in \eqref{normaliz1}.
\end{theorem}

\begin{theorem}\label{existmax}   
Let $\lambda_1(m)$ be the principal eigenvalue of problem
\eqref{p0},   
$m_0\in L^\infty(\Omega)$      
such that the set $\{x\in \Omega:m_0(x)>0\}$ has positive Lebesgue measure, $\mathcal{G}(m_0)$ the class of rearrangements
of $m_0$ (see Definition \ref{class}) and $\overline{\mathcal{G}(m_0)}$ its closure with respect to the weak*
topology of $L^\infty(\Omega)$. Then \\
i) if $\int_\Omega m_0\;dx\leq 0$, then
\begin{equation*}  
\sup_{m\in{\mathcal{G}(m_0)}} \lambda_1(m)=+\infty;
\end{equation*}
ii) if $\int_\Omega m_0\;dx>0$, then \\
\phantom{aa} a) there is a unique solution $\hat{m}_1$ of the problem
$$\sup_{m\in\overline{{\mathcal{G}(m_0)}}} \lambda_1(m);$$ 
\phantom{aa} b) in the case of Dirichlet boundary conditions, $\hat{m}_1\geq 0$ a.e.
 in $\Omega$. Moreover, if $m_0\geq 0$\\ \phantom{aa} a.e. in $\Omega$, then 
(denoting by  $\hat{m}_1^*$  and  $m_0^*$ the decreasing rearrangements 
of $\hat{m}_1$ 
 and $m_0$\\  \phantom{aa} respectively, see Definition \ref{riord}) $\hat{m}_1^*=m_0^*$,  i.e.
$\hat{m}_1\in \mathcal{G}(m_0)$; otherwise, let  $\gamma\in(0,|\Omega|)$ 
 such\\ \phantom{aa} that
  $\int_\gamma^{|\Omega|}m_0^*\;ds=0$, then
$$\hat{m}_1^*(s)=
\begin{cases}
m_0^*(s)\quad & \text{if }\ 0<s<\gamma\\
0 & \text{if }\ \gamma\leq s<|\Omega|;
\end{cases}
$$
\phantom{aa} c) in the case of Robin boundary conditions, if $m_0\geq 0$ a.e. in $\Omega$, then
$\hat{m}_1\in \mathcal{G}(m_0)$;\\
\phantom{aa} d) for both boundary conditions, there exists a decreasing function $\psi$ such
that      
\begin{equation}\label{carat2}\hat{m}_1= \psi(u_{\hat{m}_1}) \quad \text{a.e. in }\Omega, 
\end{equation}
\phantom{aa} where 
$u_{\hat{m}_1}$ is 
the unique positive eigenfunction relative to $\lambda_1(\hat{m}_1)$ normalized 
as in \eqref{normaliz1}.
\end{theorem}   

Theorem \ref{exist} and part of Theorem \ref{existmax} in the case of the Dirichlet boundary 
conditions are contained in \cite{CCP}. In this paper we give an alternative treatment by which we 
are able to handle both Dirichlet and Robin boundary conditions here and Neumann in \cite{AC23}.
The main novelty of this paper is the investigation of the maximization of the principal eigenvalue,
which is rarely discussed in literature. In particular, for the Dirichlet boundary conditions, whenever the 
maximizer exists, we explicitly describe its decreasing rearrangement.\\
From a biological point of view, Theorem \ref{exist} says that, within our class of local growth 
rates $\mathcal{G}(m_0)$, there exists at least a rearrangement $\check{m}_1$
of the resources in the environment which provides the largest range of the diffusion rate $d$ 
such that the population persists (recall that, in this case, we have $d<1/\lambda_1(m)$).
In other words,  this configuration of the resources maximizes 
the chances of survival. Moreover, by \eqref{carat0},  the more favourable the habitat is, the 
larger the population density is. In the case of the Dirichlet boundary conditions, this means that the 
the most favourable habitat must be located far from the boundary of $\Omega$.  This is perfectly consistent with the results obtained in \cite{CC,CC91} and, obviously, 
with the conclusion in \cite{CCP}.

Conversely, Theorem \ref{existmax} gives information about the best way of distributing the 
resources through the environment $\Omega$ to maximize the possibilities of extinction of the
population. A typical practical application is the problem of determining how to distribute the
pesticide (which is a negative resource) over the environment in order to have the greatest negative 
impact on the pest population when the pesticide is not enough to treat the whole region $\Omega$.
Precisely, i) says that, if the total resources are nonpositive (i.e., the total unfavourable habitat 
exceeds or is equal to the favourable one), then,  
for any given diffusion rate, we can find a rearrangement of resources which guarantees the extinction of the population (in this case, we have $d\geq 1/\lambda_1(m)$).
We underline that, in this situation, it does not exist an optimal configuration in the class $\mathcal{G}
(m_0)$, nevertheless, in some sense (see Theorem \ref{lemmafondmax} for details), the nonpositive constant
$c=\frac{1}{|\Omega|}\int_\Omega m_0\,dx$ is an optimal disposition of the resources in the extended 
class $\overline{\mathcal{G}(m_0)}$.
Therefore, a natural question arises: which rearrangements are more suitable to accomplish the extinction of the population?
Being $\lambda_1(m)$ continuous on $\overline{\mathcal{G}(m_0)}$ endowed with the weak* 
topology of $L^\infty(\Omega)$ (see Lemma \ref{teo1}), we should choose $m$ close
to $c\leq 0$. We think this implies that $m$ should be somewhere equal to $c$ and elsewhere 
highly oscillatory, which is consistent with the conclusion in \cite{CC} in the case of the Dirichlet 
boundary conditions.

Under the hypothesis ii), that is when the total resources are positive, there is a
unique disposition of the resources which gives the largest range of the diffusion rate which causes the 
extinction of the population. However, in general, this disposition is not a plain rearrangement of the 
given resources (i.e. it does not belong to $\mathcal{G}(m_0)$).\\
In the case of the Dirichlet boundary conditions, this disposition is obtained first, by using the 
negative resources to cancel out a suitable part of the positive ones, and then rearranging appropriately
what remains (this disposition belongs to $\overline{\mathcal{G}(m_0)}$).
As a consequence, it corresponds to an environment everywhere favourable
and it is a genuine rearrangement of the initial resources if and only if they are all positive (i.e. $m_0\geq 0$).
Then, we can ask: when the initial resources are both positive and negative, 
which rearrangements of them are more suitable to produce the extinction of the population?
Similarly to the case i), we should choose $m$ close
to $\hat{m}_1$. We think this implies that $m$ should be equal to $\hat{m}_1$ where $
\hat{m}_1$ is positive and highly oscillatory where $\hat{m}_1$ is equal to zero
(the more $m$ is close to $\hat{m}_1$, the higher is the frequency of the oscillations).
We will explore this issue in a future paper. Biologically, this means that a precise
share of the favourable habitat must be concentrated near the boundary (i.e. where $\hat{m}_1$
is positive) and the rest must be located far from it (where $\hat{m}_1$ is zero)  and interlaced 
as much as possible with the unfavourable habitat. In other words, we have a concentration of 
positive resources near the boundary and a highly fragmented habitat far from it.
Returning to the example of a pest population (and considering the exterior of the environment 
as hostile, i.e. Dirichlet boundary conditions), the best strategy to exterminate it is to distribute
the pesticide not across the entire environment, but rather appropriately far from its boundary.\\
The case of the Robin boundary conditions is more complex: we prove that the optimal disposition
is a plain rearrangement of the initial resources provided they are positive everywhere.
For both boundary conditions, by \eqref{carat2}, the population density is small where the habitat is
favourable. We observe that the increased complexity of this case is probably 
due to the fact that its behaviour is in between the problems with Neumann and Dirichlet conditions
(see \cite{CC91, LLNP}).\\
Incidentally, we note that the results about the minimization of the principal eigenvalue in the case of Neumann 
boundary conditions are similar to those contained in Theorem \ref{exist}, conversely the conclusions about
the maximization problem are much less involved of those in Theorem \ref{existmax} (see 
\cite{AC23}). \\
Regarding the physical interpretation of \eqref{p0}, Theorem \ref{exist} means that, in order to
obtain the lowest principal natural frequency, the more dense material should be placed far
from the boundary of the membrane. Vice versa, by Theorem \ref{existmax}, the highest principal natural frequency is 
attained building the membrane with the more dense material near
the boundary.

This paper is structured as follows. In Section \ref{preliminaries} we set the functional framework 
and some tools in order to study the spectrum of problem \eqref{p0}.
In Section \ref{rearrangements} we collect some known results about rearrangements of 
measurable functions and give an alternative proof and generalization of a lemma in \cite{B87},
which we will use to prove Theorem \ref{existmax}.        
 Finally, in Sections \ref{properties} and \ref{opt} we give the proofs of the main results.

\section{Notations, preliminaries and weak formulation of \eqref{p0}}\label{preliminaries}   
Throughout the paper, and unless otherwise specified, measurable means Lebesgue measurable and $|E|$ denotes the Lebesgue measure of a measurable set $E\subseteq \mathbb{R}^N$.\\
Let $\Omega \subset \mathbb{R}^N$, $N\geq 1$, be a bounded connected open set with Lipschitz boundary
$\partial \Omega$.\\
 We denote  by $L^
\infty(\Omega)$, $L^2(\Omega)$, $H^1(\Omega)$ and $H_0^1(\Omega)$  the usual Lebesgue and Sobolev spaces. 
The norm in $L^\infty (\Omega)$ is denoted by 
\begin{equation*}
\|u\|_{L^\infty(\Omega)}=\esssup_{\Omega} |u| \quad \forall\, u\in L^\infty(\Omega),       
\end{equation*}
and by weak* convergence we always mean the weak* convergence in $L^\infty(\Omega)$.  
The scalar product and norm in $L^2(\Omega)$ are denoted by  
\begin{equation*}
\langle u,v \rangle_{L^2(\Omega)}=\int_\Omega uv \, dx \quad \forall\, u, v\in L^2(\Omega)
\end{equation*}
and
\begin{equation*}
\|u\|_{L^2(\Omega)}=\langle u,u \rangle^{1/2}_{L^2(\Omega)} \quad \forall\, u\in L^2(\Omega)
\end{equation*}  
respectively.
Let $\sigma(x)\in C(\partial \Omega)$ and $\sigma(x)\geq 0$, as a limit case we also allow $\sigma(x)\equiv +\infty$. In $H^1(\Omega)$ we consider the scalar product 
\begin{equation*}
\langle u,v \rangle_{H^1(\Omega)}=\int_\Omega \nabla u \cdot \nabla v \, dx +\int_{\partial\Omega}\sigma uv\, dS 
\quad \forall\, u, v\in H^1(\Omega)
\end{equation*}
and the associated norm
\begin{equation}  \label{normaH1}  
\|u\|_{H^1(\Omega)}=\left(\langle u,u \rangle_{H^1(\Omega)}\right)^{1/2} \quad \forall\, u\in H^1(\Omega).
\end{equation}      
It can be shown that \eqref{normaH1} is equivalent to the usual norm (see, for example,\cite{Mik}).       
Finally, in $H_0^1(\Omega)$ we use the standard scalar product and the associated norm
\begin{equation*}
\langle u,v \rangle_{H_0^1(\Omega)}=\int_\Omega \nabla u \cdot \nabla v \, dx \quad \forall\, u, v\in H_0^1(\Omega)
\end{equation*}
and
\begin{equation*}        
\|u\|_{H_0^1(\Omega)}=\left(\langle u,u \rangle_{H_0^1(\Omega)}\right)^{1/2} \quad \forall\, u\in H_0^1(\Omega)
\end{equation*}
respectively.
\noindent In the sequel the spaces $H^1(\Omega)$ and $H_0^1(\Omega)$ will be always endowed with the scalar products and norms defined above. In order to give an unified treatment of problem \eqref{p0} with both Dirichlet and Robin boundary conditions, 
we introduce the notation 
\begin{equation*}
W_\sigma(\Omega)=\begin{cases} H^1(\Omega) \quad &\text{if } \sigma(x)\in C(\partial \Omega),\ \sigma(x)\geq 0,\ \sigma(x)\not
\equiv 0\\
 H_0^1(\Omega)  &\text{if } \sigma(x)\equiv +\infty.
\end{cases}
\end{equation*}
The right scalar product and norm in $W_\sigma(\Omega)$ will be denoted by  $\langle \cdot,\cdot \rangle_{W_\sigma(\Omega)}$ 
and $\|\cdot\|_{W_\sigma(\Omega)}$ respectively.
The case $\sigma(x)\equiv 0$, corresponding to the Neumann boundary conditions, is rather different and it has been treated in
\cite{AC23}.

\subsection{The operators $E_m$ and $G_m$}\label{operators}

We study the eigenvalues of problem \eqref{p0} by means of the spectrum of an operator that we introduce in this subsection.\\
Let $m\in L^\infty(\Omega)$. 
For every $f\in L^2(\Omega)$ let us consider the following continuous linear functional of $W_\sigma(\Omega)$ 
\begin{equation*} 
\varphi\mapsto
 \langle m f,\varphi\rangle_{L^2(\Omega)} \quad\forall\, \varphi\in W_\sigma(\Omega).
\end{equation*}
By the Riesz Theorem, 
there exists a unique $u\in W_\sigma(\Omega)$ such that 
\begin{equation}\label{volpe} 
\langle  u, \varphi\rangle_{W_\sigma(\Omega)}
= \langle m f,\varphi\rangle_{L^2(\Omega)} \quad\forall\, \varphi\in W_\sigma(\Omega)
\end{equation}
holds. 
Let us introduce the linear operator
\begin{equation*}
 E_m:L^2(\Omega)\to W_\sigma(\Omega),
\end{equation*}   
where $u=E_m(f)$ is the unique function in $W_\sigma(\Omega)$ that satisfies \eqref{volpe}, i.e.  for
all $f\in  L^2(\Omega)$, $E_m(f)$ is defined by 
\begin{equation*}  
\langle E_m(f) , \varphi\rangle_{W_\sigma(\Omega)}
= \langle m f,\varphi\rangle_{L^2(\Omega)} \quad\forall\, \varphi\in W_\sigma(\Omega).
\end{equation*}    
$E_m$ is clearly linear. Moreover,  putting $\varphi=u$ in \eqref{volpe}, using Poincar\'e inequality when $\sigma(x)\equiv
+\infty$ and the equivalence of the norm \eqref{normaH1} with the usual norm of $H^1(\Omega)$ when $\sigma(x)$ is bounded, 
we find         
\begin{equation}\label{normau}    
\|u\|_{W_\sigma(\Omega)}\leq C(\sigma)\|m\|_{L^\infty(\Omega)}\|f\|_{L^2(\Omega)},   
\end{equation}  
where $C(\sigma)$ is a constant depending on $\sigma$.\\     
The inequality \eqref{normau} implies
\begin{equation}\label{normaEm}
\|E_m\|_{\mathcal{L}(L^2(\Omega), W_\sigma(\Omega))}\leq C(\sigma)\|m\|_{L^\infty(\Omega)}.
\end{equation}   
Let $i_\sigma$ be the compact inclusion of $W_\sigma(\Omega)$ into $L^2(\Omega)$.  
Moreover, we define a further linear operator
\begin{equation}\label{Gm}   
G_m:W_\sigma(\Omega)\to W_\sigma(\Omega),  
\end{equation}  
by the composition $G_m=E_m\circ i_\sigma$, i.e. for all $f\in W_\sigma(\Omega)$, $G_m(f)$     
is given by    
\begin{equation}\label{Gmf} 
\langle G_m(f) , \varphi\rangle_{W_\sigma(\Omega)}
= \langle m f,\varphi\rangle_{L^2(\Omega)} \quad\forall\, \varphi\in W_\sigma(\Omega).
\end{equation}   
The inequality \eqref{normau} implies
\begin{equation*}\quad
\|G_m\|_{\mathcal{L}(W_\sigma(\Omega), W_\sigma(\Omega))}\leq C(\sigma)^2\|m\|_{L^\infty(\Omega)}.
\end{equation*}    
In the sequel we will use the straightforward formula
\begin{equation}\label{linearit}
G_{am +bq}= a\,  G_{m}+b\, G_q\quad \forall\, m,q\in L^\infty(\Omega), \ 
\forall\, a, b\in \mathbb{R}.
\end{equation}

\begin{proposition}\label{proprietaop}
Let $m\in L^\infty(\Omega)$ and $G_m$ be the operator
defined by \eqref{Gmf}. Then $G_m$ is self-adjoint and compact.
\end{proposition}

\begin{proof}
For all $f, g\in W_\sigma(\Omega)$, by \eqref{Gmf}, we have
\begin{equation*}
\langle G_m(f), g\rangle_{W_\sigma(\Omega)} =
\langle m f, g\rangle_{L^2(\Omega)}
= \langle m g, f\rangle_{L^2(\Omega)}
=\langle G_m (g), f\rangle_{W_\sigma(\Omega)},
 \end{equation*} 
then $G_m$ is self-adjoint.\\
 The compactness of the operator $G_m$  is an immediate consequence 
of the representation $G_m=E_m \circ i_\sigma $.  
\end{proof} 
 
By general theory of self-adjoint compact operators (see \cite{DF,Lax})
it follows that all nonzero eigenvalues of $G_m$ have a finite dimensional eigenspace 
and they can be obtained by the {\it Fischer's Principle}   
\begin{equation}\label{Fischer1}  \mu_k(m) =
\sup_{F_k\subset W_\sigma(\Omega) }
\inf_{f\in F_k\atop f\neq 0}
\cfrac{\langle G_m (f), f \rangle_{W_\sigma(\Omega)} }{\|f\|^2_{W_\sigma(\Omega)}} \,
= \sup_{F_k\subset W_\sigma(\Omega)}\inf_{f\in F_k\atop f\neq 0}
\cfrac{\int_\Omega m f^2 \, dx }{\|f\|^2_{W_\sigma(\Omega)}} \,, \quad k=1, 2, 3, \ldots\end{equation}  
and
\begin{equation*}      
\mu_{-k}(m) =  \inf_{F_k\subset W_\sigma(\Omega)}
\sup_{f\in F_k\atop f\neq 0}
\cfrac{\langle G_m (f), f \rangle_{W_\sigma(\Omega)} }{\|f\|^2_{W_\sigma(\Omega)}} \,=
\inf_{F_k\subset W_\sigma(\Omega)}
\sup_{f\in F_k\atop f\neq 0}
\cfrac{\int_\Omega m f^2 \, dx }{\|f\|^2_{W_\sigma(\Omega)}}\,,\quad k=1, 2, 3,  \ldots,
\end{equation*}      
where the first extrema are taken over all the subspaces $F_k$ of $W_\sigma(\Omega)$ of dimension
$k$. As observed in \cite{DF}, all the inf's and sup's in the above characterizations of the eigenvalues are actually
assumed. Hence, they could be replaced by min's and max's and the eigenvalues are obtained 
exactly in correspondence of the associated eigenfunctions.
The sequence $\{\mu_k(m)\}$ contains all the real positive eigenvalues (repeated with their multiplicity), is decreasing and converging to zero, whereas $\{\mu_{-k}(m)\}$ is formed by all the real negative
eigenvalues (repeated with their multiplicity), is increasing and converging to zero.\\ 
We will write $\{m>0\}$ as short form of $\{x\in \Omega: m(x)>0\}$ and similarly 
$\{m<0\}$ for $\{x\in \Omega: m(x)<0\}$. 
  
\begin{proposition}\label{segnorho}Let $m \in L^\infty(\Omega)$, $G_m$ be the operator 
\eqref{Gmf} and $\mu_k(m)$, $\mu_{-k}(m)$ its eigenvalues. 
Then, the following statements hold\\   
i) if $|\{m>0\}|=0$, then there are no positive eigenvalues;\\
ii) if $|\{m>0\}|>0$, then there is a sequence of positive 
eigenvalues $\mu_k(m)$;\\
iii) if $|\{m<0\}|=0$, then there are no negative eigenvalues;\\
iv) if $|\{m<0\}|>0$, then there is a sequence of negative 
eigenvalues $\mu_{-k}(m)$.
\end{proposition} 

\begin{proof} 
i) Let $\mu$ be an eigenvalue and $u$ a corresponding eigenfunction. By \eqref{Gmf} with 
$f=\varphi=u$ we have    
$$\mu
= \cfrac{\int_\Omega m u^2\, dx }{\|u\|^2_{W_\sigma(\Omega)}}\, \leq 0.$$    
ii) By measure theory covering theorems, for each positive integer $k$ there exist $k$ 
disjoint closed balls $B_1, \ldots, B_k$ in $\Omega$ such that $| B_i \cap \{m>0\}|>0$ 
for $i=1, \ldots, k$. Let $f_i\in C^\infty_0(B_i)$ such that 
$\int_\Omega m f_i^2 \, dx=1$ for every $i=1, \ldots, k$. 
Note that the functions  $f_i$ are linearly independent and let $F_k= \ $span$  \{f_1, \ldots, f_k\}$. $F_k$ is  
a subspace of $W_\sigma(\Omega)$. For every $f\in F_k\smallsetminus 
\{0\}$, let $f=\sum_{i=1}^k a_i f_i$, $a_i\in \mathbb{R}$. 
Then, by \eqref{Gmf} we have   
\begin{equation*} \begin{split}\cfrac{\langle G_m (f), 
f\rangle_{W_\sigma(\Omega)}}{\|f\|^2_{W_\sigma(\Omega)}}\, & =
\cfrac{\int_\Omega mf^2\, dx}{\|f\|^2_{W_\sigma(\Omega)}} \,=  
\cfrac{\sum_{i,j=1}^k a_{i}a_j \int_\Omega mf_i f_j\, dx}
{\sum_{i,j=1}^k a_{i}a_j \langle f_i, f_j\rangle_{W_\sigma(\Omega)}} \,\\ &   
=\cfrac{\sum_{i=1}^k a_i^2}{\sum_{i,j=1}^k a_{i}a_j \langle f_i, f_j\rangle_{W_\sigma(\Omega)}} \,            
=\cfrac{\|a\|^2_{\mathbb{R}^k}}{\langle A_k a, a\rangle_{\mathbb{R}^k}}\,\geq \cfrac{1}{\|A_k\|}\, >0,
\end{split}
\end{equation*}
where $\|a\|_{\mathbb{R}^k}$, $\|A_k\|$ and $\langle A_k a, a\rangle_{\mathbb{R}^k}$ denote
the euclidean norm of the vector $a=(a_1, \ldots, a_k)$, the norm of the non null matrix 
$A_k=\left( \langle f_i, f_j\rangle_{W_\sigma(\Omega)} \right)_{i, j=1}^k$ and the scalar product in $\mathbb{R}^k$, 
respectively.
From the Fischer's Principle \eqref{Fischer1} we conclude that $\mu_k(m)\geq \cfrac{1}{\|A_k\|}\, >0$ for
every $k$.\\
The cases iii) and iv) are similarly proved.
\end{proof}
In the case $\sigma(x)\equiv+\infty$ (Dirichlet boundary conditions) and for a general elliptic operator, the
previous proposition has been proved in \cite{DF}. 
 
\subsection{Weak formulation of problem \eqref{p0}} 

A function $u\in W_\sigma(\Omega)$ is said to be an \emph{eigenfunction} of \eqref{p0}  associated to the \emph{eigenvalue} $\lambda$ if 
\begin{equation}\label{FD1dautov}
\int_\Omega \nabla u\cdot \nabla\varphi \, dx = \lambda \int_\Omega m u\varphi\, dx \quad\forall\, \varphi\in H_0^1(\Omega) 
\end{equation}
for $\sigma(x)\equiv+\infty$ and if
\begin{equation}\label{FD1rautov}
\int_\Omega \nabla u\cdot \nabla\varphi \, dx +\int_{\partial \Omega} \sigma u\varphi\, dS= \lambda \int_\Omega m u\varphi\, dx \quad\forall\, \varphi\in H^1(\Omega) 
\end{equation}
for $\sigma(x)$ bounded.
Equivalently, using the scalar product in $W_\sigma(\Omega)$, \eqref{FD1dautov} and \eqref{FD1rautov} are summarized by
\begin{equation}\label{falena}
\langle  u, \varphi\rangle_{W_\sigma(\Omega)}
=\lambda \langle m u,\varphi\rangle_{L^2(\Omega)} \quad\forall\, \varphi\in W_\sigma(\Omega).
\end{equation} 
It is easy to check that zero is not an eigenvalue of \eqref{p0}.

\begin{proposition}\label{equivalenzalambdamu}
Let $m \in L^\infty(\Omega)$ and $G_m$ be
the operator \eqref{Gm}. 
Then the nonzero eigenvalues of problem \eqref{p0} are exactly the reciprocal 
of the nonzero eigenvalues of the operator $G_m$ and the correspondent eigenspaces coincide. 
\end{proposition}      
 \begin{proof}
If $\lambda\neq 0$, by \eqref{falena}  we have 
\begin{equation*}     
\left\langle \frac{u}{\lambda},\varphi\right\rangle_{W_\sigma(\Omega)}
=\langle m u,\varphi\rangle_{L^2(\Omega)} \quad\forall\, \varphi\in W_\sigma(\Omega)
\end{equation*}
and then, by definition \eqref{Gmf} of $G_m$, $G_m (u)= \cfrac{u}{\lambda}\, $. \\
Vice versa, let $G_m(u)=\mu u$, with $\mu\neq 0$. Then we have
\begin{equation*}
\left\langle \mu u,\varphi\right\rangle_{W_\sigma(\Omega)}
=\langle m u,\varphi\rangle_{L^2(\Omega)} \quad\forall\, \varphi\in W_\sigma(\Omega)
\end{equation*}
and then
\begin{equation*}
\left\langle u,\varphi\right\rangle_{W_\sigma(\Omega)}
=\frac{1}{\mu}\langle m u,\varphi\rangle_{L^2(\Omega)} \quad\forall\, \varphi\in W_\sigma(\Omega),
\end{equation*}  
i.e. $1/\mu$ is an eigenvalue of \eqref{p0}.
The claim follows.
\end{proof}     

Consequently, in general,
the eigenvalues of problem \eqref{p0} form two monotone sequences
$$ 0<\lambda_1(m)\leq \lambda_2(m)\leq\ldots\leq  \lambda_k(m)\leq \ldots$$
and
$$ \ldots\leq\lambda_{-k}(m)\leq\ldots\leq \lambda_{-2}(m)\leq  \lambda_{-1}(m)<0  ,$$
where every eigenvalue appears as many times as its multiplicity, the latter being finite
owing to the compactness of $G_m$. \\  
The variational characterization \eqref{Fischer1} for $k=1$, assuming that $|\{m>0\}|>0$, becomes  
  
\begin{equation}\label{mu1}  \mu_1(m) =  \max_{f\in W_\sigma(\Omega)\atop f\neq 0}
\cfrac{\langle G_m (f), f \rangle_{W_\sigma(\Omega)} }{\|f\|^2_{W_\sigma(\Omega)}} \,
=\max_{f\in W_\sigma(\Omega)\atop f\neq 0}
\cfrac{\int_\Omega m f^2 \, dx }{\|f\|^2_{W_\sigma(\Omega)}}.  
\end{equation}    
The maximum in \eqref{mu1} is obtained if and only if $f$ is an eigenfunction relative to $\mu_1(m)$.
From \eqref{mu1}, it follows a similar variational characterization of $\lambda_1(m)$ and an analogous remark applies. 
Thus, for $\lambda_1(m)$ we have
\begin{equation*}
\lambda_1(m) = \min_{u\in W_\sigma(\Omega)\atop \int_\Omega m u^2dx>0}
\cfrac{\|u\|^2_{W_\sigma(\Omega) }}{\int_\Omega m u^2 \, dx}\,.
\end{equation*}

\begin{proposition}\label{simple} Let $m\in L^\infty(\Omega)$ such that $|\{m>0\}|>0$. Then $\mu_1(m)$ is simple and any associated eigenfunction is one signed in $\Omega$.
\end{proposition}  

\begin{proof} Let $u\in W_\sigma(\Omega)$ be an eigenfunction related to $\mu_1(m)$. 
By \eqref{mu1}, $|u|$ is again an eigenfunction.  
By Proposition  \ref{equivalenzalambdamu}, $|u|$ satisfies the equation $-\Delta |u|
=\mu_1(m)^{-1} m|u|$ and, by Harnack inequality        
(see \cite{GT}), we conclude that $|u|>0$ in $\Omega$; therefore $u$ is one signed in $\Omega$.\
Let $u,v$ be two  eigenfunctions; set $\alpha=\frac{\int_\Omega v\,dx}{\int_\Omega u\,dx}$, note that $\int_\Omega (\alpha u- v)\, dx=0$ and $\alpha u-v$ belongs to the eigenspace of $
\mu_1(m)$. If $\alpha u-v$ was not identically zero, then, by Harnack inequality, it would be 
one signed and hence $\int_\Omega (\alpha u- v)\, dx\neq0$, reaching a contradiction. Therefore $v=\alpha u$ and thus $\mu_1(m) $ is simple.  
\end{proof}    

As an immediate consequence of Proposition \ref{simple}, we have the 
following 
  
\begin{corollary} Let $m\in L^\infty(\Omega)$ such that $|\{m>0\}|>0$. Then $\lambda_1(m)$ is simple and any associated eigenfunction is one signed in $\Omega$.    
\end{corollary}   
 
Throughout the paper we will denote by $u_m$ the unique positive eigenfunction of both $G_m$
(relative to $\mu_1(m)$)     
and problem \eqref{p0} (relative to $\lambda_1(m)$) , normalized by 
\begin{equation}\label{normaliz1}
\|u_m\|_{W_\sigma(\Omega)}= 1,
\end{equation}
which is equivalent to 
\begin{equation}\label{normaliz2}
\int_\Omega m u_m^2 \, dx= \mu_1(m)=\cfrac{1}{\lambda_1(m)}\,.
\end{equation}   
By standard regularity theory (see \cite{GT,Mik}), $u_m \in W^{2,2}_{\text{loc}}(\Omega)\cap C^{1, \beta}(\Omega)$
for all $0<\beta<1$.\\
As last comment, we observe that $\mu_1(m)$ is homogeneous of degree 1, i.e. 
\begin{equation}\label{homo}
\mu_1(\alpha m)= \alpha \mu_1(m) \quad \forall\, \alpha>0,
\end{equation}
this follows immediately from \eqref{mu1}.
\medskip

\section{Rearrangements of measurable functions}\label{rearrangements}
In this section we introduce the concept of
rearrangement of a measurable function and summarize some related results we will use in the next section. 
The idea of rearranging a function dates back to the book \cite{hardy52} 
of Hardy, Littlewood and P\'olya, since than many authors have investigated both extensions and 
applications of this notion. Here we relies on the results in \cite{Alvino1,B87,B89,day70,
Kaw,ryff67}.

Let $\Omega$ be an open bounded set of $\mathbb{R}^N$.

\begin{definition}
For every measurable 
function $f:\Omega\to\mathbb{R}$ the function $d_f:\mathbb{R}\to [0,|\Omega|]$ defined by
$$d_f(t)=|\{x\in\Omega: f(x)>t\}|$$
is called \emph{distribution function of $f$}.
\end{definition}

The symbol $\mu_f$ is also used. It is easy to prove the
following properties of $d_f$.

\begin{proposition}
For each $f$ the distribution function $d_f$ is decreasing, right continuous and
the following identities hold true
$$\lim_{t\to-\infty}d_f(t)=|\Omega|,\quad\quad \lim_{t\to\infty}d_f(t)=0.$$
\end{proposition}

\begin{definition}
Two measurable functions $f,g:\Omega \to \mathbb{R}$ are called \emph{equimeasurable} functions 
or \emph{rearrangements} of one another if one of the following equivalent conditions is satisfied

i)   $|\{x\in \Omega: f(x)>t\}|=|\{x\in \Omega: g(x)>t\}| \quad \forall\, t\in \mathbb{R}$;

ii)  $d_f=d_g$.
\end{definition}

Equimeasurability of $f$ and $g$ is denoted by $f\sim g$. Equimeasurable functions
share global extrema and integrals as it is stated precisely by the following proposition.
\begin{proposition}\label{rospo}
Suppose $f\sim g$ and let $F:\mathbb{R}\to\mathbb{R}$ be a Borel measurable function, then
 
i) $|f|\sim |g|$;

ii) $\esssup f=\esssup g$ and  $\essinf f=\essinf g$;

iii) $F\circ f\sim F\circ g$;

iv) $F\circ f\in L^1(\Omega)$ implies $F\circ g\in L^1(\Omega)$ and $\int_\Omega F\circ f\,dx=
\int_\Omega F\circ g\,dx$.

\end{proposition}
For a proof see, for example, \cite[Proposition 3.3]{day70} or \cite[Lemma 2.1]{B89}.

In particular, for each $1\leq p\leq\infty$, if $f\in L^p(\Omega)$ and $f\sim g$ then
$g\in L^p(\Omega)$ and 
\begin{equation*}
 \|f\|_{L^p(\Omega)}=\|g\|_{L^p(\Omega)}.
\end{equation*}

\begin{definition}\label{riord}
For every measurable function $f:\Omega\to\mathbb{R}$ the function $f^*:(0,|\Omega|)\to
\mathbb{R}$ defined by
$$f^*(s)=\sup\{t\in\mathbb{R}: d_f(t)>s\}$$
is called \emph{decreasing rearrangement of $f$}.
\end{definition}
An equivalent definition (used by some authors) is $f^*(s)=\inf\{t\in\mathbb{R}: d_f(t)\leq
s\}$. 

\begin{proposition} \label{furbi}
For each $f$ its decreasing rearrangement $f^*$ is decreasing, right continuous and
we have 
$$\lim_{s\to 0}f^*(s)=\esssup f\quad\text{and}\quad\lim_{s\to|\Omega|}f^*(s)=\essinf f.$$
Moreover, if $F:\mathbb{R}\to\mathbb{R}$ is a Borel measurable function then
$F\circ f\in L^1(\Omega)$ implies $F\circ f^*\in L^1(0,|\Omega|)$ and 
$$\int_\Omega F\circ f\,dx=\int_0^{|\Omega|} F\circ f^*\,ds.$$ 
Finally, $|\{x\in\Omega: f(x)>t\}|=|\{x\in\Omega: f^*(x)>t\}|$ for all $t\in\mathbb{R}$,
$d_{f^*}=d_f$ and, for each measurable function $g$ we have $f\sim g$ if and only if $f^*=g^*$.
\end{proposition}

Some of the previous claims are simple consequences of the definition of $f^*$, for
more details see \cite[Chapter 2]{day70}. 

As before, it follows that, for each $1\leq p\leq\infty$, if $f\in L^p(\Omega)$ then
$f^*\in L^p(0,|\Omega|)$ and $\|f\|_{L^p(\Omega)}=\|f^*\|_{L^p(0,|\Omega|)}.$
 
\begin{definition}\label{prec1}
Given two functions $f,g\in L^1(\Omega)$, we write $g\prec f$ if
$$\int_0^t g^*\,ds\leq \int_0^t f^*\,ds\quad\forall\; 0\leq t\leq|\Omega|\quad\quad
{and}\quad\quad\int_0^{|\Omega|} g^*\,ds= \int_0^{|\Omega|} f^*\,ds.$$ 
\end{definition}
Note that $g\sim f$ if and only if $g\prec f$ and $f\prec g$. Among many properties of the relation
$\prec$ we mention the following (a proof is in \cite[Lemma 8.2]{day70}) proposition.

\begin{proposition}\label{prec}
For any pair of functions $f,g\in L^1(\Omega)$ and real numbers $\alpha$ and $\beta$,
if $\alpha\leq f\leq\beta$ a.e. in $\Omega$ and $g\prec f$, then $\alpha\leq g\leq\beta$ a.e. in $\Omega$.
\end{proposition}

\begin{proposition}\label{prec2}
For $f\in L^1(\Omega)$ let $g=\frac{1}{|\Omega|}\int_\Omega f \, dx$. Then we have 
$g\prec f$.
\end{proposition}

\begin{definition}\label{class} 
Let $f:\Omega\to\mathbb{R}$ a measurable function. We call the set
$$\mathcal{G}(f)=\{g:\Omega\to\mathbb{R}: g \text{ is measurable and } g\sim f \}$$
\emph{class of rearrangements of $f$} or \emph{set of rearrangements of $f$}.
\end{definition}

Note that, for $1\leq p\leq\infty$, if $f$ is in $L^p(\Omega)$ then $\mathcal{G}(f)$ is also
contained in $L^p(\Omega)$.

As we will see in the next section, we are interested in the optimization of a functional 
defined on a class of rearrangements $\mathcal{G}(m_0)$,
where $m_0$ belongs to $L^\infty(\Omega)$. For this reason, although almost all of what follows
holds in a much more general context, hereafter we restrict our attention to classes of rearrangement of
functions in $L^\infty(\Omega)$.
We need compactness properties of the set $\mathcal{G}(m_0)$, with a little
effort it can be showed that this set is closed but in general it is not compact in the norm topology
of $L^\infty(\Omega)$. Therefore we focus our attention on the weak* compactness.
By $\overline{\mathcal{G}(m_0)}$ we denote the closure of $\mathcal{G}(m_0)$ in the weak* topology
of $L^\infty(\Omega)$.

\begin{proposition}\label{cane}
 Let $m_0$ be a function of $L^\infty(\Omega)$. Then $\overline{\mathcal{G}(m_0)}$ is
 
 i)  weakly* compact;
 
 ii) metrizable in the weak* topology;
 
 iii) sequentially weakly* compact. 
\end{proposition}
For the proof see \cite[Proposition 3.6]{ACF}. 

Moreover, the sets $\mathcal{G}(m_0)$ and $\overline{\mathcal{G}(m_0)}$ have further
properties. 

\begin{definition}\label{libellula}
Let $C$ be a convex set of a real vector space. An element $v$ in $C$ is said an \emph{extreme point of
$C$} if for every $u$ and $w$ in $C$ the identity $v=\frac{u+w}{2}$ implies $u=w$.
\end{definition}
A vertex of a convex polygon is an example of extreme point.

\begin{proposition}\label{convexity}
Let $m_0$ be a function of $L^\infty(\Omega)$, then

i) $\overline{\mathcal{G}(m_0)}=\{f\in L^\infty(\Omega): f\prec m_0\}$;

ii) $\overline{\mathcal{G}(m_0)}$ is convex;

iii) $\mathcal{G}(m_0)$ is the set of the extreme points of $\overline{\mathcal{G}(m_0)}$.
\end{proposition}
\begin{proof}
The claims follow from \cite[Theorems 22.13, 22.2, 17.4, 20.3]{day70}.  
\end{proof}

An evident consequence of the previous theorem is that $\overline{\mathcal{G}(m_0)}$ is
the weakly* closed convex hull of $\mathcal{G}(m_0)$.
\begin{corollary}\label{minore}

 Let $m_0\in L^\infty(\Omega)$ and $m,q\in\overline{\mathcal{G}(m_0)}$. Then\\
 i) $\int_\Omega m\,dx=\int_\Omega m_0\,dx$;\\
 ii) assuming $\int_\Omega m_0\,dx\neq 0$, $m=q$ if and only if $m$ and $q$ are linearly dependent.
\end{corollary}    
\begin{proof}
i) It follows immediately by i) of Proposition \ref{convexity}, Definition \ref{prec1} and Proposition 
 \ref{furbi} with $F$ equal to the identity function.\\
 ii) If $m$ and $q$ are linearly dependent, then, without loss of generality we can assume that
 $m=\alpha q$, for some $\alpha\in\mathbb{R}$. Integrating over $\Omega$ and using i)
 we find $m=q$.        
\end{proof}       
 Part i) of the following proposition  as been proved in \cite{B87} in a general measure 
 space, here we give an alternative proof. Moreover, part ii)  extends in some sense i) to the 
 case of sign changing functions.
\begin{proposition}\label{burtgen}
 Let $m_0\in L^\infty(\Omega)$ such that $\int_\Omega m_0\,dx>0$. Then\\
 i) if $m_0\geq0$ in $\Omega$, we have 
 \begin{equation}\label{uno}
 |\{x\in\Omega: m_0(x)>0\}|\leq|\{x\in\Omega: m (x)>0\}|\quad\forall\,m\in\overline{\mathcal{G}(m_0)}; 
 \end{equation}
 ii)
 if $|\{x\in\Omega: m_0(x)<0\}|>0$, we choose $\gamma\in (0,|\Omega|)$ such that $\int_\gamma^{|\Omega|}m_0^*
 \,ds=0$. Then, there exists a subset $E\subset\Omega$ of measure $\gamma$ such that, if $\overline{m}_0$ is
 the weight defined by 
\begin{equation}\label{barm0}
\overline{m}_0(x)=
\begin{cases}
m_0(x)\quad &\text{if } x\in E\\
 0 &\text{if }  x\in \Omega\smallsetminus E,
\end{cases}  
\end{equation}  
we have
\begin{equation}\label{mostar}
\overline{m}_0^*(s)=
\begin{cases}
m_0^*(s)\quad &\text{if } 0<s<\gamma\\
 0 &\text{if }  \gamma\leq s< |\Omega|
\end{cases}  
\end{equation}
and
\begin{equation}\label{due}
|\{x\in\Omega: \overline{m}_0(x)>0\}|\leq|\{x\in\Omega: m (x)>0\}|\quad\forall\,m\in\overline{\mathcal{G}
(m_0)},\,
m\geq 0.
\end{equation}

\end{proposition}
\begin{proof}
i) Let $m\in\overline{\mathcal{G}(m_0)}$ 
nonnegative. First observe that, by  Proposition \ref{furbi}, \eqref{uno} is equivalent to 
 \begin{equation}\label{burtdisug}
|\{s\in(0,|\Omega|): m^*_0(s)>0\}|\leq|\{s\in(0,|\Omega|): m^* (s)>0\}|,
\end{equation}
we put $\alpha=|\{s\in(0,|\Omega|): m_0^*(s)>0\}|$ and $\beta=
|\{s\in(0,|\Omega|): m^* (s)>0\}|$. Being the decreasing rearrangement a decreasing and right continuous 
function, we also have $\{s\in(0,|\Omega|): m_0^*(s)>0\}=(0,\alpha)$ and $\{s\in(0,|\Omega|): m^* (s)>0\}
=(0,\beta)$. 
Suppose now that \eqref{burtdisug} is false, i.e. $\alpha>\beta$.  By i) of Proposition \ref{convexity} and Definition \ref{prec1} we have 
$$\int_0^\beta m^*\,ds=\int_0^{|\Omega|}m^*\,ds=\int_0^{|\Omega|}m_0^*\,ds=\int_0^\alpha m_0^*\,ds>
\int_0^\beta m_0^*\,ds,$$
which contradicts $m\prec m_0$.\\
ii) Let $m\in\overline{\mathcal{G}(m_0)}$ 
nonnegative. We begin by noting that
 \begin{equation}\label{subset}
\{s\in(0,|\Omega|): m^*_0(s)>m^*_0(\gamma)\}\subset(0,\gamma)\subset\{s\in(0,|\Omega|): m_0^* 
(s)\geq m^*_0(\gamma)\}.
\end{equation}
From \eqref{subset} and Proposition \ref{furbi} we obtain
 \begin{equation*}
|\{x\in\Omega: m_0(x)>m^*_0(\gamma)\}|\leq\gamma\leq|\{x\in\Omega: m_0 
(x)\geq m^*_0(\gamma)\}|,
\end{equation*}
then we choose a subset $E$ of $\Omega$ of measure $\gamma$ such that  
 \begin{equation*}
\{x\in\Omega: m_0(x)>m^*_0(\gamma)\}\subseteq E\subseteq\{x\in\Omega: m_0 
(x)\geq m^*_0(\gamma)\}.
\end{equation*}
Let $\overline{m}_0$ defined in \eqref{barm0}, it is immediate to find 
\begin{equation*}
\{x\in\Omega: \overline{m}_0(x)>t\}=
\begin{cases}
\Omega\quad &\text{if } t<0\\
          E             &\text{if } 0\leq t<m^*_0(\gamma)\\
\{x\in\Omega: m_0(x)>t\}  &\text{if }  t\geq m^*_0(\gamma).
\end{cases}  
\end{equation*}
Hence, we have \eqref{mostar},
from which immediately follows $\overline{m}_0\prec m_0$, i.e. $\overline{m}_0\in\overline{\mathcal{G}(m_0)}$.
The proof concludes arguing exactly as in part i): we note that \eqref{due} is equivalent to 
 \begin{equation}\label{burtdisug2}
|\{s\in(0,|\Omega|): \overline{m}^*_0(s)>0\}|\leq|\{s\in(0,|\Omega|): m^*(s)>0\}|,
\end{equation}
we have $\gamma=|\{s\in(0,|\Omega|): \overline{m}_0^*(s)>0\}|$ and  put $\beta=
|\{s\in(0,|\Omega|): m^* (s)>0\}|$. We observe that $\{s\in(0,|\Omega|): \overline{m}_0^*(s)>0\}=(0,\gamma)$ and $\{s\in(0,|
\Omega|): m^* (s)>0\}=(0,\beta)$. 
Suppose \eqref{burtdisug2} is false, i.e. $\gamma>\beta$. By the same argument as in i), we have 
$$\int_0^\beta m^*\,ds=\int_0^{|\Omega|}m^*\,ds=\int_0^{|\Omega|}\overline{m}_0^*\,ds=
\int_0^\gamma\overline{m}_0^*\,ds>\int_0^\beta\overline{m}_0^*\,ds=
\int_0^\beta m_0^*\,ds,$$
which contradicts $m\prec m_0$.\\
This completes the proof.
\end{proof}

The following is \cite[Theorem 11.1]{day70} rephrased for our case.
\begin{proposition}
Let $u\in L^1(\Omega)$ and $m_0\in L^\infty(\Omega)$. Then 
\begin{equation}\label{day}
\int_0^{|\Omega|} m_0^*(|\Omega|-s) u^*(s)\, ds\leq\int_\Omega m\,u\, dx
\leq\int_0^{|\Omega|} m_0^*(s) u^*(s)\, ds\quad\quad\forall\, m\in\mathcal{G}(m_0);
\end{equation}
moreover, both sides of \eqref{day} are taken on.
\end{proposition}

The previous proposition implies that the linear optimization problems
\begin{equation}\label{max}
 \sup_{m \in \mathcal{G}(m_0)} \int_\Omega m u \, dx\end{equation}
 and
 \begin{equation}\label{minG}
 \inf_{m \in \mathcal{G}(m_0)} \int_\Omega m u \, dx\end{equation}
 admit solution.

Finally, we recall the following results proved in \cite[Theorem 3 and Theorem 5]{B87}
and \cite[Lemma 2.9]{B89} respectively.

\begin{proposition}\label{Teobart87bis}
Let $u\in L^1(\Omega)$ and $m_0\in L^\infty(\Omega)$. 
If there exists an increasing function $\psi$ such that $\psi \circ u$ is a rearrangement of $m_0$ a.e. in 
$\Omega$, then problem \eqref{max}, with $\mathcal{G}(m_0)$ replaced by $\overline{\mathcal{G}(m_0)}$,
has the unique solution $m_M:=\psi \circ u$.
\end{proposition}

\begin{proposition}\label{Teobart87}
Let $u\in L^1(\Omega)$ and $m_0\in L^\infty(\Omega)$. If problem \eqref{max} has a unique solution 
$m_M$, 
then there exists an increasing function $\psi$ such that $m_M=\psi \circ u$ a.e. in 
$\Omega$.
\end{proposition}  

\begin{proposition}\label{Teobart89}
Let $E$ be a measurable set of finite measure, $u,m:E\to\mathbb{R}$  measurable 
functions and suppose that every level set of $u$ has zero measure. Then there is an 
increasing function $\psi$ such that $\psi\circ u$ is a rearrangement of $m$. 
\end{proposition}  
  
\begin{remark}\label{ossburt}
Considering $-u$ in the previous three propositions, it is immediate to conclude that
they have three counterparts where \eqref{max} and ``increasing'' are replaced by
\eqref{minG} and ``decreasing'' respectively.
 
\end{remark}

\section{Qualitative properties of $\mu_1(m)$}\label{properties}   

In this section we will prove some qualitative properties of the eigenvalue $\mu_1(m)$
of the operator $G_m$ defined in \eqref{Gmf}.
We begin by proving the continuity of $\mu_1(m)$ (actually, of all of the eigenvalues of the 
operator $G_m$) and then showing its convexity and G\^{a}teaux 
differentiability.

Observe that, by Proposition \ref{segnorho}, 
$\mu_k(m)$ and $u_m$ (the unique positive eigenfunction of $\mu_1(m)$ of problem \eqref{p0}
normalized as in \eqref{normaliz1}) are well defined only when $|\{m>0\}|>0$. We extend 
them
to the whole space $L^\infty(\Omega)$ by putting, for $k=1,2,3,\ldots$,
\begin{equation} \label{muk}
\widetilde{\mu}_k(m)=\begin{cases} \mu_k(m) \quad & \text{if } |\{m>0\}|>0\\
0 & \text{if } |\{m>0\}|=0\end{cases}
\end{equation}
and
\begin{equation}  \label{um} 
\widetilde{u}_m=\begin{cases}u_m \quad & \text{if } |\{m>0\}|>0\\
0 & \text{if } |\{m>0\}|=0.\end{cases}
\end{equation}     

\begin{remark}\label{oss1}
Note that $\widetilde{\mu}_k(m)=0$ if and only if $|\{m>0\}|=0$ and, in this 
circumstance,  the inequality
\begin{equation}\label{tilde}\sup_{F_k\subset W_\sigma(\Omega)}
\min_{f\in F_k\atop f\neq 0}
\cfrac{\langle G_m (f), f \rangle_{W_\sigma(\Omega)} }{\|f\|^2_{W_\sigma(\Omega)}} \,\leq 0 
\end{equation}
holds, where $F_k$ varies among all the $k$-dimensional subspaces of $W_\sigma(\Omega)$. 
Moreover,  from \eqref{homo}, we have $\widetilde{\mu}_1(\alpha m)=\alpha \widetilde{\mu} 
_1(m)$ for every $m\in L^\infty(\Omega)$ and $\alpha\geq 0$.
\end{remark}

\begin{lemma}\label{teo1} 
Let $m\in L^\infty(\Omega)$,  $G_m$ be the linear operator \eqref{Gm}, $\widetilde{\mu}_k(m)$       
as defined in \eqref{muk} for $k=1,2,3,\ldots$ and $\widetilde{u}_m$ as in \eqref{um}. 
Then\\
i) the map $m\mapsto G_m$ is sequentially weakly* continuous from $L^\infty(\Omega)$
to $\mathcal{L}(W_\sigma(\Omega),W_\sigma(\Omega)) $ endowed with the norm topology;\\
ii) the map $m\mapsto \widetilde{\mu}_k(m)$ is sequentially weakly* continuous in
$L^\infty(\Omega)$; \\
iii) the map $m\mapsto \widetilde{\mu}_1(m)\widetilde{u}_m$ is sequentially weakly* continuous
from $L^{\infty}(\Omega)$ to $W_\sigma(\Omega)$ (endowed with the norm topology). In particular, 
for any sequence $\{m_i\}$ weakly* convergent to $m\in L^\infty(\Omega)$, with $\widetilde{\mu}_1(m)>0$, then 
$\{\widetilde{u}_{m_i}\}$ converges to $\widetilde{u}_{m}$ in $W_\sigma(\Omega)$.
\end{lemma} 
  
\begin{proof}
i) Let $\{m_i\}$ be a sequence which weakly* converges to $m$ in $L^\infty(\Omega)$. 
Being $\{m_i\}$ bounded in $L^\infty(\Omega)$, there exists a constant $M>0$ such that
\begin{equation}\label{M}
\|m\|_{L^\infty(\Omega)}\leq M\quad \text{ and } \quad \|m_i\|_{L^\infty(\Omega)}\leq M
\quad \forall\, i. 
\end{equation} 
We begin by proving that            
$G_{m_i}(f)$ tends to $G_m(f)$ in $W_\sigma(\Omega)$ for any fixed $f\in W_\sigma(\Omega)$.
Let $u_i=G_{m_i}(f)$ and 
$u=G_{m}(f)$.\\ 
First, we show that $u_i$ weakly converges to $u$ in $W_\sigma(\Omega)$; indeed, by \eqref{Gmf} we have             
\begin{equation*}
\langle u_i,\varphi\rangle_{W_\sigma(\Omega)}=
\langle m_if,\varphi\rangle_{L^2(\Omega)}\quad \forall\varphi \in W_\sigma(\Omega)
\end{equation*}
and
\begin{equation*}
\langle u,\varphi\rangle_{W_\sigma(\Omega)}=
\langle mf,\varphi\rangle_{L^2(\Omega)}\quad \forall\varphi \in W_\sigma(\Omega).
\end{equation*}
Subtracting, we get        
\begin{equation}\label{sottrazione}
\langle u_i-u,\varphi\rangle_{W_\sigma(\Omega)}=
\langle m_if-m f,\varphi\rangle_{L^2(\Omega)}\quad \forall\varphi \in W_\sigma(\Omega).
\end{equation}
As a consequence of the weak* convergence of $m_i$ to $m$ in $L^\infty(\Omega)$, letting $i\to
\infty$ we obtain $\int_\Omega m_i f\varphi \,dx \to \int_\Omega m f\varphi \,dx$, which implies that  the right hand term goes to zero,  thus $u_i$ weakly converges to $u$ in
$W_\sigma(\Omega)$.
By exploiting the compactness of the inclusion $W_\sigma(\Omega)\hookrightarrow L^2(\Omega)$, we deduce that  $u_i$ strongly converges to $u$ in $L^2(\Omega)$.\\
We show now that $u_i$ strongly converges to $u$ in
$W_\sigma(\Omega)$.
Indeed, putting $\varphi=u_i-u$ in \eqref{sottrazione}, by using \eqref{M} we get
\begin{equation*}\quad
\begin{split} 
 \| u_i-u\|^2_{W_\sigma(\Omega)}&=
\langle m_if-mf,u_i-u\rangle_{L^2(\Omega)}
= \int_\Omega (m_i -m)f (u_i-u)\, dx\\ &\leq2M\| f\|_{L^2(\Omega)} \| u_i-u\|_{L^2(\Omega)}.
\end{split} 
\end{equation*}   
 Summarizing, for every  $f\in W_\sigma(\Omega)$ we have   
\begin{align*}
\|G_{m_i}(f) - G_m (f)\|_{W_\sigma(\Omega)}\to 0\quad\text{for }i\to\infty.  
\end{align*}

Now, for fixed $i$, let $\{f_{i,j}\}$, $j=1,2, 3,\ldots$, be a maximizing sequence of
$$\sup_{g\in W_\sigma(\Omega)\atop \|g\|_{W_\sigma(\Omega)}\leq 1}
 \|G_{m_i}(g) - G_{m}(g)\|_{W_\sigma(\Omega)}=\|G_{m_i}-G_m\|_{\mathcal{L}(W_\sigma(\Omega), W_\sigma(\Omega))} .$$      
Then, being $\|f_{i,j}\|_{W_\sigma(\Omega)}\leq 1$, we can extract a subsequence (still denoted by $\{f_{i,j}\}$)
weakly convergent to some $f_i\in W_\sigma(\Omega)$. Since  the operators
$G_{m_i}$ and $G_m$ are compact (see Proposition 
\ref{proprietaop}), it follows that
$G_{m_i}(f_{i,j})$ converges to $G_{m_i}(f_{i})$ and  
$G_{m}(f_{i,j})$ converges to $G_{m}(f_{i})$ strongly in $W_\sigma(\Omega)$ as $j$ goes to
$\infty$. Thus we find
$$ \|G_{m_i}-G_m\|_{\mathcal{L}(W_\sigma(\Omega), W_\sigma(\Omega))}
=\|G_{m_i}(f_i) - G_m(f_i)\|_{W_\sigma(\Omega)}.$$
This procedure yields a sequence $\{f_i\}$ in $W_\sigma(\Omega)$ such that $\|f_{i}\|_{W_\sigma(\Omega)}\leq 1$ 
for all $i$. Then, up to a subsequence, we can assume that
$\{f_i\}$ weakly converges  to a function $f\in W_\sigma(\Omega)$ 
and (by compactness of the inclusion $i_\sigma: W_\sigma(\Omega)\hookrightarrow L^2(\Omega)$) strongly in $L^2(\Omega)$.  
Recalling that $G_m=E_m\circ i_\sigma$ and by using \eqref{normaEm} and \eqref{M} 
we find   
\begin{align*} & \|G_{m_i}- G_m\|_{\mathcal{L}(W_\sigma(\Omega), W_\sigma(\Omega))}=
\|G_{m_i}(f_i) - G_m(f_i)\|_{W_\sigma(\Omega)}\\
& \leq \|G_{m_i}(f) - G_m(f)\|_{W_\sigma(\Omega)} + \|G_{m_i}(f_i-f)\|_{W_\sigma(\Omega)} +\|G_m(f_i-f)\|_
{W_\sigma(\Omega)}\\  
&\leq \|G_{m_i}(f) - G_m(f)\|_{W_\sigma(\Omega)} + \|E_{m_i}(f_i-f)\|_{W_\sigma(\Omega)} +\|E_m(f_i-f)\|_
{W_\sigma(\Omega)}\\    
&\leq \|G_{m_i}(f) - G_m(f)\|_{W_\sigma(\Omega)}+\left(\|E_{m_i}\|_{\mathcal{L}(L^2(\Omega), W_\sigma(\Omega))}+\|E_{m}\|_{\mathcal{L}(L^2(\Omega), 
W_\sigma(\Omega))} \right)   
\|f_i-f\|_{L^2(\Omega)} \\      
 &\leq \|G_{m_i}(f) - G_m(f)\|_{W_\sigma(\Omega)}+2C(\sigma)M \|f_i-f\|_{L^2(\Omega)}.         
 \end{align*}    
Therefore $G_{m_i}$ converges to $G_m$ in the operator norm.\\
ii) Let $\{m_i\}$ be a sequence which weakly* converges to $m$ in $L^\infty(\Omega)$. 
Being $\{m_i\}$ bounded in $L^\infty(\Omega)$, there exists a constant $M>0$ such that
\eqref{M} holds. We show that, for any $i$ and $k=1, 2, 3, \ldots$ the estimate
\begin{equation}\label{modulo}
 |\widetilde{\mu}_k(m_i)- \widetilde{\mu}_k(m)|\leq  \|G_{m_i}-G_m\|_{\mathcal{L}(W_\sigma(\Omega), W_\sigma(\Omega))}  
\end{equation}  
holds, then the claim will follow by i).\\
We split the argument in the following three cases, in each of them $i$ and $k$ are fixed. 
 
\emph{Case 1.}  $\widetilde{\mu}_k(m_i)$, $\widetilde{\mu}_k(m)>0$.\\ Following  the 
idea in \cite[Theorem 2.3.1]{He} and by means of the Fischer's Principle 
\eqref{Fischer1} we have 
\begin{equation*} \begin{split}
\widetilde{\mu}_k(m_i)- \widetilde{\mu}_k(m) &=  \max_{F_k\subset W_\sigma(\Omega)}
\min_{f\in F_k\atop f\neq 0}
\cfrac{\langle G_{m_i}( f), f \rangle_{W_\sigma(\Omega)} }{\|f\|^2_{W_\sigma(\Omega)}} \,
- \max_{F_k\subset W_\sigma(\Omega)}
\min_{f\in F_k\atop f\neq 0}
\cfrac{\langle G_m( f), f \rangle_{ W_\sigma(\Omega)} }{\|f\|^2_{ W_\sigma(\Omega)}} \,\\
& \leq 
\min_{f\in F_k(m_i)\atop f\neq 0}
\cfrac{\langle G_{m_i} (f), f \rangle_{W_\sigma(\Omega)} }{\|f\|^2_{W_\sigma(\Omega)}} \,
- 
\min_{f\in F_k(m_i)\atop f\neq 0}
\cfrac{\langle G_m (f), f \rangle_{W_\sigma(\Omega)} }{\|f\|^2_{W_\sigma(\Omega)}} \,\\
& \leq 
\cfrac{\langle G_{m_i} (f_m), f_m \rangle_{W_\sigma(\Omega)} }{\|f_m\|^2_{W_\sigma(\Omega)}} \,-
\cfrac{\langle G_m (f_m), f_m \rangle_{W_\sigma(\Omega)} }{\|f_m\|^2_{W_\sigma(\Omega)}}\,   \\
& =
\cfrac{\langle (G_{m_i}-G_m)( f_m), f_m \rangle_{W_\sigma(\Omega)} }{\|f_m\|^2_{W_\sigma(\Omega)}} \,
 \leq \|G_{m_i}-G_m\|_{\mathcal{L}(W_\sigma(\Omega), W_\sigma(\Omega))},
\end{split} 
\end{equation*}
where $F_k(m_i)$ is a $k$-dimensional subspace of $W_\sigma(\Omega)$ such that   
$$\max_{F_k\subset W_\sigma(\Omega)}
\min_{f\in F_k\atop f\neq 0}
\cfrac{\langle G_{m_i}( f), f \rangle_{W_\sigma(\Omega)} }{\|f\|^2_{W_\sigma(\Omega)}} \, =\min_{f\in F_k(m_i)\atop f\neq 0}
\cfrac{\langle G_{m_i}( f), f \rangle_{W_\sigma(\Omega)} }{\|f\|^2_{W_\sigma(\Omega)}} \,$$
and $f_m$ is a function in $F_k(m_i)$ such that 
$$ \min_{f\in F_k(m_i)\atop f\neq 0}   
\cfrac{\langle G_m (f), f \rangle_{W_\sigma(\Omega)} }{\|f\|^2_{W_\sigma(\Omega)}} \,=
\cfrac{\langle G_m(f_m), f_m \rangle_{W_\sigma(\Omega)} }
{\|f_m\|^2_{W_\sigma(\Omega)}} \,.
$$      
Interchanging the role of $m_i$ and $m$ we find \eqref{modulo}. 

\emph{Case 2.} $\widetilde{\mu}_k(m_i)>0$, $ \widetilde{\mu}_k(m)=0$ (and similarly in the case $\widetilde{\mu}_k(m)>0$, $ \widetilde{\mu}_k(m_i)=0$).\\
Note that in this case \eqref{tilde} holds for the weight function $m$. Then
the previous argument still applies provided that we replace the first step of the inequality chain by 
\begin{equation*} |\widetilde{\mu}_k(m_i)- \widetilde{\mu}_k(m) |= 
\widetilde{\mu}_k(m_i) \leq \max_{F_k\subset W_\sigma(\Omega)}
\min_{f\in F_k\atop f\neq 0}
\cfrac{\langle G_{m_i}( f), f \rangle_{W_\sigma(\Omega)} }{\|f\|^2_{W_\sigma\Omega)}} \,
- \sup_{F_k\subset W_\sigma(\Omega)}
\min_{f\in F_k\atop f\neq 0}
\cfrac{\langle G_m( f), f \rangle_{W_\sigma(\Omega)} }{\|f\|^2_{W_\sigma(\Omega)}} \,.
\end{equation*}
  
\emph{Case 3.} $\widetilde{\mu}_k(m_i)= \widetilde{\mu}_k(m)=0$.\\
In this case \eqref{modulo} is obvious.\\  
Therefore statement i) is proved.
        
iii) Let $\{m_i\}, m$ be such that $m_i$ is weakly$^*$ convergent
to $m\in L^\infty(\Omega)$. Being $\|\widetilde{u}_{m_i}\|_{W_\sigma(\Omega)} 
\leq 1$,           
up to a subsequence we can assume that $\widetilde{u}_{m_i}$ is weakly convergent to $z
\in W_\sigma(\Omega)$, strongly in $L^2(\Omega)$ and pointwisely a.e. in $\Omega$. \\
When $\widetilde{\mu}_1(m)=0$ the assertion is straightforward. Indeed, we have
$$
\|\widetilde{\mu}_1(m_i) \widetilde{u}_{m_i}-\widetilde{\mu}_1(m) \widetilde{u}_{m}
\|_{W_\sigma(\Omega)}
=\widetilde{\mu}_1(m_i)\| \widetilde{u}_{m_i}\|_{W_\sigma(\Omega)}\leq\widetilde{\mu}_1(m_i)
$$
which, together with ii), gives the claim.

Now let us consider the case $\widetilde{\mu}_1(m)>0$. By ii) we have $\widetilde{\mu}
_1(m_i)>0$
for all $i$ large enough. This implies  $\widetilde{\mu}_1(m_i)
=\frac{1}{\lambda_1(m_i)}\, $ and $\widetilde{u}_{m_i}= u_{m_i}$. Positiveness and pointwise convergence of $u_{m_i}$ to $z$ imply $z\geq 0$ a.e. in $\Omega$.
Moreover, by \eqref{normaliz2} we have
$$ \int_\Omega m_i u^2_{m_i} \, dx =\cfrac{1}{\lambda_1(m_i)}\, $$
and by ii), passing to the limit, we find 
$$ \int_\Omega m z^2 \, dx =\cfrac{1}{\lambda_1(m)}\, ,$$
which implies $z\not\equiv 0$. By using \eqref{falena} for $u_{m_i}$  
we have
\begin{equation*}\langle  u_{m_i}, \varphi\rangle_{W_\sigma(\Omega)} =   
\lambda_1(m_i) \langle m_i u_{m_i}, \varphi\rangle_{L^2(\Omega)} = \lambda_1(m_i)
\int_\Omega  m_i u_{m_i} \varphi \, dx \quad \forall\, \varphi \in W_\sigma(\Omega)
\end{equation*} and, letting $i$ to $\infty$, we deduce $z=u_m$.
By ii) $\mu_1(m_i) u_{m_i}$     
weakly converges in $W_\sigma(\Omega)$ to $\mu_1(m)u_{m}$ and 
 $\|\mu_1(m_i)u_{m_i}\|_{W_\sigma(\Omega)}
=\mu_1(m_i)$ tends to $\mu_1(m) =\|\mu_1(m) u_{m}\|_{W_\sigma(\Omega)}$.
Hence $\mu_1(m_i) u_{m_i}$
strongly converges to $\mu_1(m)u_{m}$ in $W_\sigma(\Omega)$.
The last claim is immediate provided one observes that $\widetilde{\mu}_1(m)>0$ 
implies $\widetilde{\mu}_1(m_i)>0$ for all $i$ large enough.
\end{proof}

\begin{lemma}\label{teo2} Let $m, q, m_0\in L^\infty(\Omega)$ and   
$\widetilde{\mu}_1(m)$ be defined as in \eqref{muk} for $k=1$. Then\\
i) the map $m\mapsto \widetilde{\mu}_1(m)$     
 is convex on $L^\infty(\Omega) $; \\
ii) if $m$ and  $q$ are linearly independent and $ \widetilde{\mu}_1(m),  \widetilde{\mu}_1(q)>0$, then 
\begin{equation*} \widetilde{\mu}_1(tm+(1-t)q)< t \widetilde{\mu}_1(m)+(1-t)  \widetilde{\mu}_1(q)  
\end{equation*}
for all $0<t<1$;\\
iii) if $\int_\Omega m_0 \, dx >0$, then the map 
$m\mapsto \widetilde{ \mu}_1(m)$ 
 is strictly convex on $\overline{\mathcal{G}(m_0)}\, $;\\
 iv) if $\int_\Omega m_0\, dx \leq 0$ (and $m_0\not\equiv 0$), then the map $m\mapsto \widetilde{\mu}_1(m)$ is not
strictly convex on $\overline{\mathcal{G}(m_0)}$.       
\end{lemma}

\begin{proof}
i) The Fischer's Principle \eqref{Fischer1} and \eqref{tilde} both for $k=1$ yield
\begin{equation}\label{basta}
\sup_{f\in W_\sigma(\Omega) \atop f\neq 0}
\cfrac{\langle G_m(f),f\rangle_{W_\sigma(\Omega)}}{ \|f\|^2_{W_\sigma(\Omega)}}\,\leq \widetilde{ \mu}_1(m)  
\end{equation}
for every $m \in L^\infty(\Omega)$. Moreover, if $\widetilde{ \mu}_1(m)>0$, then the equality sign holds and  
the supremum is attained when $f$ is an eigenfunction of $\widetilde{\mu}_1(m)=\mu_1(m)$.
Let $m, q\in L^\infty(\Omega)$, $0\leq t\leq 1$. We show that  
\begin{equation}\label{conv}
\widetilde{\mu}_1(tm + (1-t)q)\leq t \widetilde{\mu}_1 (m) +(1-t) \widetilde{\mu}_1 (q).
\end{equation}
If $\widetilde{\mu}_1(tm + (1-t)q)=0$ then, \eqref{conv} is obvious. Suppose 
$\widetilde{\mu}_1(tm + (1-t)q)>0$. Then, for all $f\in W_\sigma(\Omega)$, $f\neq 0$, we 
have  
\begin{equation} \label{A1} 
\cfrac{\langle G_{tm+(1-t)q} (f),f\rangle_{W_\sigma(\Omega)}}{ \|f\|^2_{W_\sigma(\Omega)}}\,
=t \cfrac{\langle G_{m}(f),f\rangle_{W_\sigma(\Omega)}}{ \|f\|^2_{W_\sigma(\Omega)}
}\,+
(1-t)\cfrac{\langle G_{q}(f),f\rangle_{W_\sigma(\Omega)}}{ \|f\|^2_{W_\sigma(\Omega)}}\,
\leq\, t\widetilde{\mu}_1 (m) +(1-t) \widetilde{\mu}_1 (q),        
\end{equation}
where we used \eqref{linearit} and \eqref{basta} for $m$ and $q$. Taking the supremum in the left-hand term  of \eqref{A1}   
and using \eqref{basta} again with equality sign,  
we find \eqref{conv}.
  
ii) Arguing by contradiction, we suppose that 
equality holds in \eqref{conv}. We will conclude that $m$ and $q$ are linearly dependent. 
Equality sign in \eqref{conv} implies   
 $\widetilde{\mu}_1(tm + (1-t)q)>0$, then (by \eqref{basta}) the equality also holds in \eqref{A1} with 
$f=u=u_{tm+(1-t)q}$.  We get 
$$\cfrac{\langle G_{m} (u),u\rangle_{W_\sigma(\Omega)}}{ \|u\|^2_{W_\sigma(\Omega)}}\,
= \widetilde{\mu}_1(m) \quad \text{ and }\quad 
\cfrac{\langle G_{q} (u),u\rangle_{W_\sigma(\Omega)}}{ \|u\|^2_{W_\sigma(\Omega)}}\,
= \widetilde{\mu}_1(q).$$
 The simplicity 
of the principal eigenvalue, the positiveness of $u$ and the normalization \eqref{normaliz1} 
imply that    
$u=u_m=u_q$.  
By using  \eqref{falena} with $\lambda=\frac{1}{\tilde{\mu}_1(m)}$ and $\lambda=\frac{1}{\tilde{\mu}_1(q)}$ we have
$$ \langle  u, \varphi \rangle_{W_\sigma(\Omega)} =\cfrac{1}{ \widetilde{\mu}_1(m)}\, 
\langle m u, \varphi\rangle_{L^2(\Omega)} \quad \forall\, \varphi \in W_\sigma(\Omega) 
$$
and
$$\langle u, \varphi\rangle_{W_\sigma(\Omega)} =  \cfrac{1}{ \widetilde{\mu}_1(q)}\, 
\langle q u, \varphi\rangle_{L^2(\Omega)} \quad \forall\, \varphi\in W_\sigma(\Omega),$$
respectively.
Taking the difference of these identities we find 
$$\left\langle \left( \cfrac{m}{ \widetilde{\mu}_1(m)}\, -\cfrac{q}{ \widetilde{\mu}_1(q)}\, \right) u, \varphi\right\rangle_{L^2(\Omega)} =0 \quad \forall\, \varphi\in W_\sigma(\Omega), $$
which gives $m\widetilde{\mu}_1(q)-q\widetilde{\mu}_1(m)=0$, i.e.   
 $m$ and $q$ are linearly dependent.
 
 iii) First, note that, by i) of Corollary \ref{minore}, $\int_\Omega m\, dx=\int_\Omega m_0\, dx>0$ 
 for any $m \in \overline{\mathcal{G}(m_0)}$. Therefore, we have $|\{m>0\}| >0$ and thus
 $\widetilde{\mu}_1 (m)>0$
for all $m \in \overline{\mathcal{G}(m_0)}$. Next, recalling that  $\overline{\mathcal{G}
(m_0)}$ is convex (see ii) of Proposition \ref{convexity}), using ii) of Corollary 
\ref{minore} and the  statement ii), then  iii) follows.

iv) Applying Proposition \ref{prec2}, we deduce that the constant function $c=\frac{1}{|
\Omega|}\int_\Omega m_0\,dx$ is in $\overline{\mathcal{G}(m_0)}$. Then, $tm_0+(1-
t)c\in \overline{\mathcal{G}(m_0)}$ for every $t\in [0,1]$. We discuss
the two cases $\int_\Omega m_0\, dx=0$, $m_0\not\equiv 0$ and $\int_\Omega m_0\, dx<0$ separately.

If $\int_\Omega m_0\, dx=0$, $m_0\not\equiv 0$, then $c=0$. Thus, 
$tm_0\in \overline{\mathcal{G}(m_0)}$ for every $t\in [0,1]$ and, by Remark \ref{oss1}, we have 
$\widetilde{\mu}_1(tm_0)=t\widetilde{\mu}_1(m_0)$, which excludes strict convexity.     

We now turn to the case $\int_\Omega m_0\, dx<0$, i.e. $c<0$. From the inequality
$$tm_0+(1-t)c\leq t\|m_0\|_{L^\infty(\Omega)}+(1-t)c,$$
we obtain       
$$tm_0+(1-t)c\leq 0\text{ in }\Omega \quad \forall\, t\leq \frac{c}{c-\|m_0\|_{L^\infty(\Omega)}}.$$
Note that $c/\!\!\left(c-\|m_0\|_{L^\infty(\Omega)}\right)\in (0,1)$. 
Therefore, by \eqref{muk}, we conclude that $\widetilde{\mu}_1(m)=0$ for any 
$m$ in the closed line segment, contained in $\overline{\mathcal{G}(m_0)}$, that joins $c$ and $$\frac{c}{c-\|m_0\|_{L^\infty(\Omega)}}\,m_0+\left(1-\frac{c}{c-\|m_0\|_{L^\infty(\Omega)}}\right)c=\frac{\|m_0\|_{L^\infty(\Omega)}-m_0}{\|m_0\|_{L^\infty(\Omega)}-c}\,c.$$
This shows that the map $m\mapsto \widetilde{\mu}_1(m)$ is not
strictly convex also in this case.
\end{proof}

For the definitions and some basic results on the G\^ateaux differentiability
we refer the reader to \cite{ET}.

\begin{lemma}\label{teo3} 
Let $m\in L^\infty(\Omega)$,  $\widetilde{\mu}_1(m)$ be
 defined as in \eqref{muk} for $k=1$ and $u_m$ denote the relative unique positive 
eigenfunction of problem \eqref{p0} normalized as in \eqref{normaliz1}.
Then, the map $m\mapsto  \widetilde{\mu}_1(m)$ 
is G\^ateaux differentiable at any $m$ such that $ \widetilde{\mu}_1(m)>0$, 
with G\^ateaux differential equal to $u_m^2$. In other words,
 for every direction $v\in L^\infty(\Omega)$  
  we have  
\begin{equation}\label{gatto} \widetilde{\mu}_1'(m; v) =\int_\Omega u_m^2 v\, dx. \end{equation}
\end{lemma}
 
\begin{proof}
Let us compute 
$$\lim_{t\to 0} \cfrac{ \widetilde{\mu}_1(m+ t v)- 
\widetilde{\mu}_1(m)}{t}\, .$$
Note that, by ii) of Lemma 
\ref{teo1}, $\widetilde{\mu}_1(m+ t v)$ converges to
$\widetilde{\mu}_1(m)>0$ as $t$ goes to zero for any $v\in L^\infty(\Omega)$. Therefore,  $\widetilde{\mu}_1(m+ t v)>0$ for $|t|$ small enough.  
The eigenfunctions $u_{m}$ and $u_{m+tv}$ satisfy (see \eqref{falena})  
$$\widetilde{\mu}_1(m)\langle u_m, \varphi\rangle_{W_\sigma(\Omega)} =
\langle m u_m, \varphi\rangle_{L^2(\Omega)} \quad \forall\, \varphi\in W_\sigma(\Omega) $$
and   
$$\widetilde{\mu}_1(m+ t v)\langle u_{m+tv}, \varphi\rangle_{W_\sigma(\Omega)} =
\langle (m+tv) u_{m+tv}, \varphi\rangle_{L^2(\Omega)} \quad \forall\, \varphi\in 
W_\sigma(\Omega). $$ 
By choosing $\varphi=u_{m+tv}$ in the former equation, $\varphi=u_m$ in the latter
and comparing we get     
\begin{equation*}
\widetilde{\mu}_1(m+ t v)
\langle m u_m, u_{m+tv}\rangle_{L^2(\Omega)}=
\widetilde{\mu}_1(m) 
\langle (m+tv) u_{m+tv}, u_m\rangle_{L^2(\Omega)}.
\end{equation*}
Rearranging we find
\begin{equation}\label{rap} \cfrac{\widetilde{\mu}_1(m+ t v)- 
\widetilde{\mu}_1(m)}{t}\, \int_\Omega m \, u_m
u_{m +tv}\, dx =\widetilde{\mu}_1(m) \int_\Omega  u_m
u_{m +tv} v\, dx .\end{equation}
If $t$ goes to zero,  
then by iii) of Lemma \ref{teo1} it follows that  
$u_{m+ t v}$ converges to $u_m$ in $W_\sigma(\Omega)$ and therefore
in $L^2(\Omega)$. Passing to the limit      
in \eqref{rap} and using \eqref{normaliz2} we conclude
\begin{equation*}
\lim_{t\to 0} \cfrac{\widetilde{\mu}_1(m+ t v)- \widetilde{\mu}_1(m)}{t}\, =
\int_\Omega u_m^2 v\, dx,
\end{equation*}
i.e. \eqref{gatto} holds.
\end{proof}

\section{Optimization of $\lambda_1(m)$}\label{opt}
This section is devoted to the study of the optimization of $\lambda_1(m)$. For this purpose, 
we will use some qualitative properties of $\mu_1(m)=1/\lambda_1(m)$ with respect to $m$, proved in the 
previous section.

\begin{theorem}\label{lemmafond}
Let $m_0\in L^\infty(\Omega)$, $\overline{\mathcal{G}(m_0)}$ be the weak* closure in $L^\infty(\Omega)$ of 
the class of rearrangements $\mathcal{G}(m_0)$ introduced in Definition \ref{class} and $\widetilde{\mu}_
1(m)$ defined as in \eqref{muk} for $k=1$. Then \\ 
i) there exists a solution of the problem 
\begin{equation}\label{infclos}
\max_{m\in\overline{ \mathcal{G}(m_0)}}\tilde{\mu}_1(m);  
\end{equation} 
ii) if $|\{m_0>0\}|>0$ (note that, in this case, by Proposition \ref{segnorho}
$\tilde\mu_1(\check m_1)=\mu_1(\check m_1)>0$), then\\
\phantom{a} a) any solution $\check m_1$ of \eqref{infclos} belongs to $\mathcal{G}(m_0)$, more 
explicitly, we have $  \tilde\mu_1(m)<\tilde\mu_1(\check m_1)$ for\\ \phantom{a} all $m\in\overline{\mathcal{G}(m_0)}
\smallsetminus
\mathcal{G}(m_0)$; \\     
\phantom{a} b) for every solution $\check{m}_1\in\mathcal{G}(m_0)$ of \eqref{infclos}
 there exists an increasing function $\psi$ such  
that           
\begin{equation*}\label{carat}\check{m}_1= \psi(u_{\check{m}_1}) \quad \text{a.e. in }\Omega, \end{equation*} \phantom{a} where 
$u_{\check{m}_1}$ is 
the positive eigenfunction relative to $\mu_1(\check{m}_1)$ normalized 
as in \eqref{normaliz1}.    
\end{theorem}    

\begin{proof}    
i) By iii) of Proposition \ref{cane} and ii) of Lemma \ref{teo1}, $\overline{\mathcal{G}(m_0)}$ 
is sequentially weakly* compact and the map $m \mapsto \widetilde{\mu}_1(m)$
is sequentially weakly* continuous, respectively. Therefore, there exists $\check{m}_1\in 
\overline{\mathcal{G}(m_0)}$ such that 
\begin{equation*}  
\widetilde{\mu}_1(\check{m}_1)=\max_{m\in \overline{\mathcal{G}(m_0)}}\widetilde{ \mu}_1(m) .
\end{equation*} 
ii) a) Note that, by Proposition \ref{segnorho}, 
the condition $|\{m_0>0\}|>0$ guarantees $\widetilde{\mu}_1(m)>0$ on 
$\mathcal{G}(m_0)$ and then $\widetilde{\mu}_1(\check{m}_1)>0$.   
Let $\check{m}_1$ be an arbitrary solution of \eqref{infclos}, let us show
 that $\check{m}_1$ actually belongs to $\mathcal{G}(m_0)$. 
 Proceeding by contradiction,  
 suppose that $\check{m}_1\not \in\mathcal{G}(m_0) $. Then, by iii) of Proposition
 \ref{convexity} and by Definition \ref{libellula}, $\check{m}_1$ is not an extreme point of $\overline{\mathcal{G}(m_0)}$
 and thus there exist $m, q \in \overline{\mathcal{G}(m_0)}$ such that $m\neq q$
 and $\check{m}_1= \frac{m + q}{2}\, $.
 By i) of Lemma \ref{teo2} and, from the fact that $\check{m}_1$ a maximizer of \eqref{infclos},
 we have 
$$\widetilde{\mu}_1(\check{m}_1)\leq   
 \cfrac{\widetilde{\mu}_1(m) +\widetilde{\mu}_1(q)}{2}\,
\leq \widetilde{\mu}_1(\check{m}_1)$$ and
then, equality signs hold. This implies  
$\widetilde{\mu}_1(m) =\widetilde{\mu}_1(q)=\widetilde{\mu}_1(\check{m}_1)>0$,   
that is $m$ and $q$ are maximizers as well. Now, applying ii) of Lemma \ref{teo2}
to $m$ and $q$ with $t=\frac{1}{2}\, $, we conclude that $m$ and $q$ are linearly   
dependent. Without loss of generality, we can assume that there exists $\alpha \in \mathbb{R}$ such 
that $q =\alpha m$, moreover $\alpha$ is nonzero since $q$ is a maximizer. Combining
 $q =\alpha m$ with $\check{m}_1= \frac{m + q}{2}\, $ we get 
 $\check{m}_1= \frac{1+\alpha}{2}\, m=\frac{1+\alpha}{2\alpha}\, q$. It is immediate
 to show that at least one of the coefficients $\frac{1+\alpha}{2}\,$ and $\frac{1+\alpha}{2\alpha}\,$
 must be nonnegative. In either cases we find a contradiction. For instance, if $\frac{1+\alpha}{2\alpha}\,\geq 0$, by Remark \ref{oss1} and maximality of $q$ we obtain 
 $$\widetilde{\mu}_1(\check{m}_1)= \cfrac{1+\alpha}{2\alpha}\,\widetilde{\mu}_1 (q)=
 \cfrac{1+\alpha}{2\alpha}\,\widetilde{\mu}_1 (\check{m}_1),$$
 which implies $\alpha=1$ and yields the contradiction $m=q$. The other case is analogous.\\
Thus, we conclude that $\check{m}_1\in \mathcal{G}(m_0)$ and a) of ii) is proved.\\      
ii) b) Let $\check{m}_1\in \mathcal{G}(m_0)$ a solution of \eqref{infclos}. We prove the claim by using Proposition \ref{Teobart87}; more precisely, we show that   
\begin{equation}\label{silvia}        
\int_\Omega 
\check{m}_1 u_{\check{m}_1}^2\, dx> \int_\Omega 
m\, u_{\check{m}_1}^2\, dx\quad\forall\, m\in \overline{\mathcal{G}(m_0)}\smallsetminus \{\check{m}_1\}.
\end{equation} 
By exploiting the convexity of $\widetilde{\mu}_1(m)$ (see Lemma \ref{teo2}) and its 
G\^ateaux differentiability in $\check{m}_1$ (see Lemma \ref{teo3})
 we have (for details see \cite{ET})
 \begin{equation} \label{maria}
\widetilde{\mu}_1(m)\geq     
\widetilde{\mu}_1\big(\check{m}_1)
+\int_\Omega (m-
\check{m}_1) u_{\check{m}_1}^2\, dx
\end{equation} 
for all $m\in \overline{\mathcal{G}(m_0)}$.
  
First, let us suppose $\widetilde{\mu}_1(m)< \widetilde{\mu}_1(\check{m}_1)$.
Comparing with \eqref{maria} we find  
\begin{equation*}
\int_\Omega 
(m-\check{m}_1) u_{\check{m}_1}^2\, dx<0  ,\end{equation*}
that is \eqref{silvia}.\\    
Next, let us consider the case $\widetilde{\mu}_1(m)=\widetilde{\mu}_1(\check{m}_1)$,
$m\in \overline{\mathcal{G}(m_0)}\smallsetminus \{\check{m}_1\}$. By a) of ii) there are not maximizers of $\widetilde{\mu}_1$ in $\overline{\mathcal{G}(m_0)}\smallsetminus \mathcal{G}(m_0)$, therefore $m \in \mathcal{G}(m_0)$.    
If $\check{m}_1$ and $m$ are linearly independent, then, ii) of Lemma \ref{teo2}
implies 
$$ \widetilde{\mu}_1 \left(\frac{\check{m}_1 + m}{2} \right)<\frac{\widetilde{\mu}_1(
\check{m}_1) +\widetilde{\mu}_1(m)}{2}\, = \widetilde{\mu}_1(
\check{m}_1).$$
Then, as in the previous step, \eqref{maria} with $\frac{\check{m}_1 + m}{2}\,$ in place of 
$m$ yields \eqref{silvia}.\\
Finally, let $\check{m}_1$ and $m$ be linearly dependent. Being $\check{m}_1$ and 
$m$ both nonzero, we can assume $m=\alpha \check{m}_1$
for a constant $\alpha\in \mathbb{R}$. Therefore $|m|= |\alpha|\,  |\check{m}_1|$. Now, 
by i) and ii) of Proposition \ref{rospo}, the functions $|m|$ and $| \check{m}_1|$
are equimeasurable and $\esssup |m|= \esssup |\check{m}_1|>0$.
This leads to $|\alpha|=1$ and, being $m$ and $\check{m}_1$ distinct, $\alpha=-1$. Thus $m=
-\check{m}_1$, which by \eqref{normaliz2} gives 
$$\int_\Omega 
m\,  u_{\check{m}_1}^2\, dx=-\int_\Omega \check{m}_1 
u_{\check{m}_1}^2\, dx = -  \widetilde{\mu}_1(\check{m}_1)<  \widetilde{\mu}_1(
\check{m}_1) = \int_\Omega \check{m}_1
\,  u_{\check{m}_1}^2\, dx, $$
i.e. \eqref{silvia}.
This completes the proof.
\end{proof}  

{\begin{remark} If $m_0$ satisfies the stronger condition $\int_\Omega m_0\, dx >0$, then the proof of part ii) 
simplifies as one can rely on iii) of Lemma \ref{teo2} (strict convexity of $m \mapsto 
\widetilde{\mu}_1(m)$ in $\overline{\mathcal{G}(m_0)}$). Indeed, if $\check{m}_1=\frac{m+q}{2}\, $, $m,q\in\overline{\mathcal{G}(m_0)}$, $m\neq q$, 
it follows immediately the contradiction  
$$ \widetilde{\mu}_1(\check{m}_1)< \cfrac{\widetilde{\mu}_1(m)
 +\widetilde{\mu}_1(q)}{2}\, \leq  \widetilde{\mu}_1(\check{m}_1).$$
Further, note that in this case $\lambda_1(m)=1/\mu_1(m)$ is well defined for all $m\in 
\overline{\mathcal{G}(m_0)}$ (it follows by i) of Corollary \ref{minore} and Proposition \ref{segnorho}).\\
\end{remark}

\begin{theorem}\label{lemmafondmax} 
Let $m_0\in L^\infty(\Omega)$, $\overline{\mathcal{G}(m_0)}$ be the weak* closure in $L^\infty(\Omega)$ of 
the class of rearrangements $\mathcal{G}(m_0)$ introduced in Definition \ref{class} and $\widetilde{\mu}
_1(m)$  defined as in \eqref{muk} for $k=1$. Then \\ 
i) there exists a solution of the problem 
\begin{equation}\label{infclos02}
\min_{m\in\overline{ \mathcal{G}(m_0)}}\tilde{\mu}_1(m);  
\end{equation}
ii) if $\int_\Omega m_0\;dx\leq 0$, then
\begin{equation*}  
\min_{m\in\overline{ \mathcal{G}(m_0)}}\tilde{\mu}_1(m)=\inf_{m\in{\mathcal{G}(m_0)}} \tilde\mu_1(m)=0
\end{equation*}
\phantom{ii)}and the constant weight $c=\frac{1}{|\Omega|}\int_\Omega m_0\;dx$ is a minimizer in 
$\overline{\mathcal{G}(m_0)}$;\\
iii) if $\int_\Omega m_0\;dx>0$, then \\
\phantom{aa} a) there is a unique solution $\hat{m}_1$ of the problem \eqref{infclos02} and, 
additionally, $\tilde{\mu}_1(\hat{m}_1)>0$;\\ 
\phantom{aa} b) in the case of Dirichlet boundary conditions,  $\hat{m}_1\geq 0$ a.e. 
in $\Omega$. Moreover, if $m_0\geq 0$ a.e. \phantom{aa} in $\Omega$, then $\hat{m}
_1^*=m_0^*$, i.e. $\hat{m}_1\in \mathcal{G}(m_0)$; otherwise, let  $\gamma\in(0,|\Omega|)$ such 
that $\int_\gamma^{|\Omega|}m_0^*\;ds=0$, \phantom{aa}  then
$$\hat{m}_1^*(s)=
\begin{cases}
m_0^*(s)\quad & \text{if }\ 0<s<\gamma\\
0 & \text{if }\ \gamma\leq s<|\Omega|;
\end{cases}
$$
\phantom{aa} c) in the case of Robin boundary conditions, if $m_0\geq 0$ a.e. in $\Omega$, then $\hat{m}_1\in \mathcal{G}(m_0)$;\\
\phantom{aa} d) for both boundary conditions, there exists a decreasing function $\psi$ such
that      
\begin{equation*}
\hat{m}_1= \psi(u_{\hat{m}_1}) \quad \text{a.e. in }\Omega, 
\end{equation*} 
\phantom{aa} where 
$u_{\hat{m}_1}$ is 
the positive eigenfunction relative to $\mu_1(\hat{m}_1)$ normalized 
as in \eqref{normaliz1}.
\end{theorem}  

\begin{proof} 
i) We repeat the reasoning used to prove i) of Theorem \ref{lemmafond}.\\
ii) Note that, by Proposition \ref{prec2} and i) of Proposition \ref{convexity},
the nonpositive constant function $c=\frac{1}{|\Omega|}\, \int_\Omega m_0 \, dx$ belongs to 
$\overline{\mathcal{G}(m_0)}$. Therefore, by definition of $\widetilde{\mu}_1(m)$, 
$\min_{m \in \overline{\mathcal{G}(m_0)}} \widetilde{\mu}_1(m)=0$. Then, being $
\mathcal{G}(m_0)$ dense in $\overline{\mathcal{G}(m_0)}$ and 
$\widetilde{\mu}_1(m)$
sequentially weak* continuous, it follows that
$\inf_{m \in \mathcal{G}(m_0)} \widetilde{\mu}_1(m)=0$.\\
iii) a) Observe that, being $\int_\Omega m_0\;dx>0$,  by i) of Corollary \ref{minore}, we have $|\{m>0\}|>0$ for all $m\in\overline{\mathcal{G}(m_0)}$.  Hence, by Proposition 
\ref{segnorho}, $\widetilde{\mu}_1(m)>0$ (and thus $\widetilde{\mu}_1(m)=\mu_1(m)$) for all $m\in\overline{\mathcal{G}(m_0)}$. 
Moreover, by iii) of Lemma \ref{teo2}, the map $m 
\mapsto\widetilde{\mu}_1(m)$ is  strictly convex. Therefore, there 
exists a unique $\hat{m}_1\in 
\overline{\mathcal{G}(m_0)}$ such that 
\begin{equation}\label{minhat}
\tilde\mu_1(\hat{m}_1)=\min_{m\in \overline{\mathcal{G}(m_0)}}\tilde\mu_1(m)>0.
\end{equation}
For the convenience of the reader, we split the proof of iii) into some steps.

iii) b) \\
\emph{Step 1.} Let $m\in \overline{\mathcal{G}(m_0)}$, by \eqref{minhat}  we have (recall that 
$\overline{\mathcal{G}(m_0)}$ is convex)
\begin{equation*}
\cfrac{\tilde\mu_1(\hat{m}_1+ t (m-\hat{m}_1))-\tilde\mu_1(\hat{m}_1)}{t}>0
\end{equation*}   
for all $t\in (0, 1)$.     
Passing to the limit for $t\to 0$ and by virtue of Lemma   \ref{teo3}  we find
\begin{equation}\label{linprob}
 \int_\Omega \hat{m}_1 u^2_{\hat{m}_1}\, dx \leq \int_\Omega m u^2_{\hat{m}_1}\, dx     
\quad \forall\, m \in\overline{\mathcal{G}(m_0)}.    
\end{equation} 
\emph{Step 2.}
We follow the idea in \cite{BML}. Let us introduce the sets 
$$P=\{x\in\Omega: \hat{m}_1(x)>0\},\quad N=\{x\in\Omega: \hat{m}_1(x)<0\},\quad Z=
\{x\in\Omega: \hat{m}_1(x)=0\}$$     
and show that $N$ has zero measure.
Define $p=\esssup_P u^2_{\hat{m}_1}$. We prove by contradiction that 
$u^2_{\hat{m}_1}\geq p$ a.e. in $N\cup  Z$. Suppose this is false, then there exist a
subset $A\subset N\cup Z$  of
positive measure and $\epsilon>0$ such that $u^2_{\hat m_1}<p-2\epsilon$ in $A$.
Let $B$ be a subset of $P$  of positive measure such that  $u^2_{\hat m_1}>p-\epsilon$
in $B$.
 Without loss of generality, we can assume $|B|=|A|$.   Let $\pi: A\to B$ be a measure preserving bijection, i.e. a bijective map such that for every set $E\subset A$,  $E$ is measurable if and only if $\pi(E)$ is measurable and in that case $|\pi(E)|=|E|$ (for more  details and the proof of
 the existence of such a map we refer the reader to \cite{B87,B89} and references 
 therein).
 We define the weight
 \begin{equation*}
  \tilde m=
  \begin{cases}
     \hat m_1(\pi(x))\quad &\text{if } x\in A\\
     \hat m_1(\pi^{-1}(x))&\text{if } x\in B\\
     \hat m_1(x)                &\text{if } x\in \Omega\smallsetminus(A\cup B).
 \end{cases}
\end{equation*}
Note that $\tilde m\sim \hat m_1\prec m_0$, i.e. $\tilde m\in\overline{\mathcal{G}(m_0)}$. 
We can write
\begin{equation}\label{assu}
\begin{split}
 \int_\Omega \hat{m}_1 u^2_{\hat{m}_1}\, dx - \int_\Omega \tilde m u^2_{\hat{m}_1}\, dx
 & = \int_{A\cup B} (\hat{m}_1-\tilde m) u^2_{\hat{m}_1}\, dx\\
 &= \int_A (\hat{m}_1-\tilde m) 
 u^2_{\hat{m}_1}\, dx +  \int_B (\hat{m}_1-\tilde m) u^2_{\hat{m}_1}\, dx\\
 &=  \int_A (\hat{m}_1-\tilde m) 
 u^2_{\hat{m}_1}\, dx +  \int_A (\tilde m-\hat{m}_1) (u^2_{\hat{m}_1}\circ\pi)\, dx\\
 &= \int_A (\tilde m-\hat{m}_1)(u^2_{\hat{m}_1}\circ\pi-u^2_{\hat{m}_1})\, dx\\
 &\geq\epsilon \int_A (\tilde m-\hat{m}_1)\, dx>0,
\end{split}                                   
\end{equation}   
where we used the fact that $(\hat{m}_1-\tilde m) u^2_{\hat{m}_1}$  restricted to $B$ and 
$(\tilde m-\hat{m}_1) (u^2_{\hat{m}_1}\circ\pi)$ restricted to $A$ have the same decreasing rearrangement and then, by Proposition \ref{furbi}, their integral coincide.
Inequality \eqref{assu} contradicts \eqref{linprob} and proves $u^2_{\hat{m}_1}\geq p$ a.e. in $N\cup  Z$.\\
\emph{Step 3.} Actually,  $u^2_{\hat{m}_1}\equiv p$ in $N\cup  Z$. Indeed,
suppose the open set $C=\{x\in\Omega:u_{\hat{m}_1}>\sqrt{p} \}$ is not empty.  Clearly,
$C\subseteq N\cup Z$. Then, $u_{\hat{m}_1}$ satisfies the boundary value problem
\begin{equation*}
\begin{cases}-\Delta u_{\hat{m}_1} =\lambda_1(\hat{m}_1) \hat{m}_1 u_{\hat{m}_1}
\quad &\text{in } C\\
 u_{\hat{m}_1}=\sqrt{p}  &\text{on } \partial C,
\end{cases}  
\end{equation*}  
where $\lambda_1(\hat{m}_1)=1/\tilde\mu_1(\hat{m}_1)$.
Being $\lambda_1(\hat{m}_1) \hat{m}_1 u_{\hat{m}_1}\leq 0$ in $C$, by the maximum 
principle we obtain the contradiction $u_{\hat{m}_1}\leq \sqrt{p}$ in $C$. Then
$C=\emptyset$ and $u^2_{\hat{m}_1}\equiv p$ in $N\cup  Z$.  Assume now $|N|>0$. This hypothesis, by using the equation $-\Delta u_{\hat{m}_1} =\lambda_1(\hat{m}_1) \hat{m}_1 u_{\hat{m}_1}$ a.e. in $N$, leads to
the contradiction  $0=\lambda_1(\hat{m}_1) \hat{m}_1 u_{\hat{m}_1}<0$. Hence 
$|N|=0$ and, consequently, $\hat{m}
_1\geq 0$  a.e. in $\Omega$.\\    
 Next, to complete the proof of iii) b), we distinguish the two cases $m_0$ nonnegative and $m_0$ 
which changes sign.\\   
\emph{Step 4.} When $m_0$ is nonnegative, by i) of Proposition \ref{burtgen} we have 
$|\{x\in\Omega: m_0(x)>0\}|\leq|P|$.  Now, note that the level sets of $u_{\hat{m}_1}$ restricted to 
$P$ have zero measure, otherwise one would obtain a contradiction with the equation  $ -\Delta 
u_{\hat{m}_1} =\lambda_1(\hat{m}_1) \hat{m}_1 u_{\hat{m}_1}$ (which holds a.e. in $\Omega$). 
In particular, $P=\{x\in\Omega: u_{\hat{m}_1}(x)<\sqrt{p}\}$.
Let us choose a  subset $F$ of $P$ such that $|F|=|\{x\in\Omega: m_0(x)>0\}|$ 
and introduce a measure preserving bijection $
\theta:F\to\{x\in\Omega: m_0(x)>0\}$. Then,  we define the weight   
\begin{equation*}
  \overline{\overline m}_0(x)=
  \begin{cases}
    m_0(\theta(x))\quad &\text{if } x\in F\\
     0                                        &\text{if } x\in \Omega\smallsetminus F.
 \end{cases}            
\end{equation*}      
We have $ \overline{\overline m}_0\sim m_0$, i.e. 
$\overline{\overline{m}}_0\in\mathcal{G}(m_0)$.
Therefore, by Proposition \ref{Teobart89} (see also Remark \ref{ossburt}) there 
exists a decreasing function $\psi:(0,\sqrt{p})\to\mathbb{R}$ such that $\psi\circ u_{\hat{m}_1}
\sim \overline{\overline m}_0$ in $P$. Defining $\psi$ equal to zero in $[\sqrt{p}, +\infty)$ we obtain $
\psi\circ u_{\hat{m}_1}\sim\overline{\overline m}_0\sim m_0$ in the whole $\Omega$. 
By Proposition \ref{Teobart87bis}, $\psi\circ u_{\hat{m}_1}$ is the unique minimizer of problem 
\eqref{linprob}, therefore we conclude that $\hat{m}_1=\psi\circ u_{\hat{m}_1}
\sim m_0$. This proves the claim and d) when $m_0$ is 
nonnegative.\\
 Now, suppose $m_0$ changes sign in $\Omega$. \\
\emph{Step 5.} By using ii) of Proposition \ref{burtgen}  we find a subset $E$ of $\Omega$ of measure
$\gamma$ such that the weight $\overline{m}_0$ defined as in \eqref{barm0} has the properties \eqref{mostar} and \eqref{due}.        
In particular, \eqref{mostar} and i) of Proposition \ref{convexity} give  $\overline{m}_0\in\overline{\mathcal{G}
(m_0)}$, i.e. $\overline{m}_0\prec m_0$, which implies $\overline{\mathcal{G}(\overline{m}_0)}
\subseteq\overline{\mathcal{G}(m_0)}$.  We show that $\hat{m}_1\in\overline{\mathcal{G}(\overline{m}_0)}$, i.e. $\hat{m}_1\prec \overline{m}_
0$.
If $t<\gamma$, then by \eqref{mostar}, we have            
$$\int_0^t \hat{m}_1^*\,ds\leq\int_0^t m^* _0\,ds=
\int_0^t\overline{m}^* _0\,ds,$$
while, when $t\geq\gamma$, being $\hat{m}_1^*$ nonnegative and $\hat{m}_1, 
\overline{m}_0\prec m_0$, we obtain 
$$\int_0^t \hat{m}_1^*\,ds\leq\int_0^{|\Omega|} \hat{m}_1^*\,ds=
\int_0^{|\Omega|}m^* _0\,ds=\int_0^{|\Omega|}\overline{m}^*
_0\,ds=\int_0^t\overline{m}^*_0\,ds,$$
with the equality sign when $t=|\Omega|$.
Then $\hat{m}_1\in\overline{\mathcal{G}(\overline{m}_0)}$.\\
The rest of the proof proceeds almost identically to the case $m_0$ nonnegative.
Nevertheless, for the convenience of the reader we repeat the argument.\\
By ii) of Proposition \ref{burtgen} and $\hat{m}_1 (x)\geq 0$ a.e. in $\Omega$, 
we have $|\{x\in\Omega: \overline{m}_0(x)>0\}|\leq|P|$.          
As before, the level sets of $u_{\hat{m}_1}$ restricted to 
$P$ have zero measure and $P=\{x\in\Omega: u_{\hat{m}_1}(x)<\sqrt{p}\}$.
Let us choose a subset $F$ of $P$ such that
$|F|=|\{x\in\Omega: \overline{m}_0(x)>0\}|$,  introduce  a  measure preserving bijection $
\theta:F\to
\{x\in\Omega: \overline{m}_0(x)>0\}$  and define   
\begin{equation*}
  \overline{\overline m}_0(x)=
  \begin{cases}
    \overline m_0(\theta(x))\quad &\text{if } x\in F\\
     0                                        &\text{if } x\in \Omega\smallsetminus F.
 \end{cases}            
\end{equation*}      
Note that $ \overline{\overline m}_0\sim\overline m_0$, thus $\overline{\mathcal{G}
(\overline{\overline{m}}_0)}=\overline{\mathcal{G}(\overline{m}
_0)}\subseteq\overline{\mathcal{G}(m_0)}$ and $\hat{m}_1\in\overline{\mathcal{G}
(\overline{\overline{m}}_0)}$. By 
Proposition \ref{Teobart89}, there 
exists a decreasing function $\psi:(0,\sqrt{p})\to\mathbb{R}$ such that $\psi\circ u_{\hat{m}_1}\sim 
\overline{\overline m}_0$ in $P$. Putting $\psi$ equal to zero in $[\sqrt{p},+\infty)$ we
obtain $\psi\circ u_{\hat{m}_1}\sim\overline{\overline m}_0\sim \overline{m}_0$ in $\Omega$. By 
Proposition \ref{Teobart87bis},
$\psi\circ u_{\hat{m}_1}$ is the unique minimizer of problem \eqref{linprob} restricted to 
$\overline{\mathcal{G}(\overline{\overline{m}}_0)}$. Since $\hat{m}_1\in\overline{\mathcal{G}
(\overline{\overline{m}}_0)}$, we conclude that $\hat{m}_1=\psi\circ u_{\hat{m}_1}
\sim \overline{m}_0$. This proves the claim and d) when 
$m_0$ changes sign.\\
iii) c) By Proposition \ref{prec}, the assumption $m_0\geq 0$ a.e. in $\Omega$ guarantees
$\hat{m}_1\geq 0$ a.e. in $\Omega$. Then, recalling the notation 
$P=\{x\in\Omega:\hat m_1(x)>0 \}$, the claim follows by using \emph{Step 1}, \emph{Step 2} and \emph{Step 4}.
This also proves d) in the case of Robin boundary conditions.\\
The proof is completed.
\end{proof}  

We are now able to prove Theorem \ref{exist} and Theorem \ref{existmax}.
\begin{proof}[Proof of Theorem \ref{exist} and \ref{existmax}]         
Being $|\{m_0>0\}|>0$, by \eqref{muk} we have  
$$\lambda_1(m)= \cfrac{1}{\mu_1(m)}\, =\cfrac{1}{\widetilde{\mu}_1(m)}\, $$
for all $m\in\mathcal{G}(m_0)$. Therefore, Theorem \ref{exist} and Theorem \ref{existmax} follow from Theorem 
\ref{lemmafond} and Theorem \ref{lemmafondmax} respectively.\\
\end{proof}    

\noindent \textbf{Acknowledgments}.
The authors are partially supported by the research project {\em Analysis of PDEs in connection 
with real phenomena}, CUP F73C22001130007, funded by \href{https://
www.fondazionedisardegna.it/}{Fondazione di Sardegna}, annuity 2021.   
The authors are members of GNAMPA (Gruppo Nazionale per l'Analisi Matematica, la 
Probabilit\`a e le loro Applicazioni) of INdAM (Istituto Nazionale di Alta Matematica
``Francesco Severi").

\end{document}